\documentclass[11pt]{article}

\usepackage[a4paper, left=3.3cm,top=3cm,right=3.3cm,bottom=3.7cm]{geometry}

\usepackage{concmath}
\usepackage[T1]{fontenc}

\usepackage{siunitx}
\usepackage{authblk}
\usepackage{psfrag}
\usepackage{amsmath,amssymb,amsthm}
\usepackage{booktabs}

\usepackage[hyphens]{url}
\usepackage{hyperref}
\usepackage{breakurl}

\usepackage{mathtools}
\usepackage{psfrag}
\usepackage{amssymb}
\usepackage{amsmath}
\usepackage{graphicx}
\usepackage{xcolor}
\usepackage[color=green!40]{todonotes}
\usepackage{enumitem}

\newtheorem{theorem}{Theorem}

\newtheorem{remark}[theorem]{Remark}

\theoremstyle{definition}

\newtheorem{definition}[theorem]{Definition}

 \newtheorem{alg}{Algorithm}

\usepackage{algorithmic}
\usepackage[ruled]{algorithm}

\DeclareMathOperator*{\argmin}{arg\,min}

\DeclareMathOperator{\sinc}{sinc}

\newcommand{\R}{\mathbb R}

\newcommand{\N}{\mathbb N}

\newcommand\edot{\,\cdot\,}

\newcommand{\Fo}{\mathcal  F}
\newcommand{\Wo}{\mathbf W}
\newcommand{\Ko}{\mathbf K}

\newcommand{\supp}{\mbox{supp}}
\newcommand{\rf}{\mathbf G}

\newcommand{\X}{ \mathcal X}
\newcommand{\Y}{ \mathcal Y}

\newcommand{\Om}{\Omega}
\newcommand{\cT}{ \mathcal T}

\newcommand{\rmd}{\mathrm d}

\newcommand{\eps}{\epsilon}
\newcommand{\coloneqq}{:=}

\newcommand{\la}{\lambda}

\newcommand{\trans}{\mathsf{T}}

\newcommand\abs[1]{\left\vert#1\right\vert}
\newcommand\sabs[1]{\vert#1\vert}
\newcommand\norm[1]{\left\Vert#1\right\Vert}
\newcommand\snorm[1]{\Vert#1\Vert}
\newcommand{\enorm}{\left\|\;\cdot\;\right\|}
\newcommand\set[1]{\left\{#1\right\}}
\newcommand\kl[1]{\left(#1\right)}

\newcommand{\fnum}{{\tt f}}
\newcommand{\pnum}{{\tt p}}
\newcommand{\qnum}{{\tt q}}
\newcommand{\gnum}{{\tt g}}
\newcommand{\unum}{{\tt u}}

\newcommand{\Wnum}{\boldsymbol{\tt{W}}}

\newcommand{\Dnum}{\boldsymbol{\tt{D}}}
\newcommand{\Tnum}{\boldsymbol{\tt{\Phi}}}
\newcommand{\interior}{\operatorname{Int}}

\newcommand{\mO}{ \mathcal O}
\newcommand{\pd}{{\partial}}

\newcommand{\dom}{\operatorname{Dom}}

\begin{document}

\title{Reconstruction algorithms for photoacoustic tomography in
heterogenous damping media}

\author{Linh V. Nguyen}

\affil{Department of Mathematics, University of Idaho\authorcr
              875 Perimeter Dr, Moscow, ID 83844, USA\authorcr
              \tt{lnguyen@uidaho.edu}             }

\author{Markus~Haltmeier}

\affil{Department of Mathematics, University of Innsbruck\authorcr
Technikerstrasse 13, 6020 Innsbruck, Austria\authorcr
\tt{markus.haltmeier@uibk.ac.at}  }

\date{August 19, 2018}

\maketitle

\begin{abstract}
In this article, we study several reconstruction methods for the inverse source problem of photoacoustic tomography (PAT) with spatially variable sound speed and damping. The backbone of these methods is the adjoint operators, which we thoroughly analyze  in both the $L^2$- and $H^1$-settings. They are casted in the form of a nonstandard wave equation. We derive the well-posedness of the aforementioned wave equation in a natural functional space, and also prove the finite speed of propagation. Under the uniqueness and visibility condition, our formulations of the standard iterative reconstruction methods, such as  Landweber's and conjugate gradients (CG), achieve a linear rate of convergence in either $L^2$- or $H^1$-norm.  When the visibility condition is not satisfied, the problem is severely ill-posed and one must apply a regularization technique to stabilize the solutions. To that end, we study two classes of regularization methods: (i) iterative, and (ii) variational regularization. In the case of full data, our simulations show that the CG method works best; it is very fast and robust. In the ill-posed case, the CG method behaves unstably.  Total variation regularization method (TV), in this case, significantly improves the reconstruction quality.

\end{abstract}

\section{Introduction}
\label{intro}
Photoacoustic tomography (PAT) is an emerging hybrid method of imaging that combines the high contrast of optical imaging  with the good resolution of ultrasound tomography. As illustrated in Figure \ref{fig:pat}, the  biological
object of interest is scanned with  a pulsed optical illumination. The
photoelastic effect causes a thermal expansion and a subsequent ultrasonic  wave propagating in space.
One measures the  ultrasonic pressure on an observation surface outside of the object.  The aim of PAT  is to recover the initial pressure distribution inside the tissue from the measured data. The initial pressure distribution
 contains helpful internal information of the object and is the image to be reconstructed.

\begin{figure}[tbh!]
\centering
  \includegraphics[width=\columnwidth]{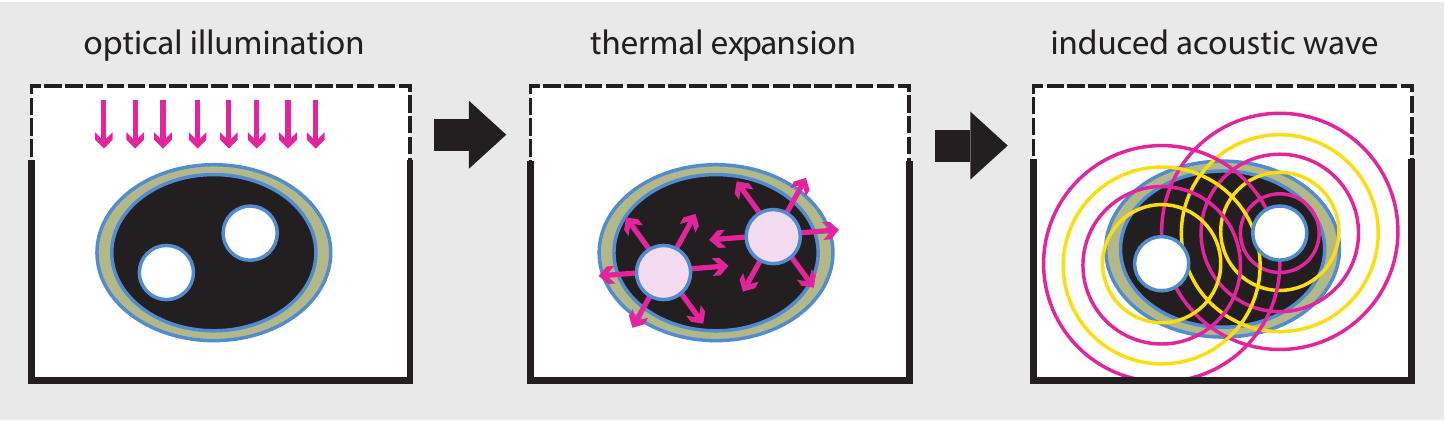}
\caption{\mbox{Left:}\label{fig:pat}  A biological object is illuminated with an optical pulse. \mbox{Middle:} Absorption of optical energy 
 causes thermal expansion.  \mbox{Right:} Thermal expansion induces an ultrasonic  wave that is measured outside of the sample 
 and used to reconstruct the image of the object.}
\end{figure}

 The standard model in PAT assumes homogeneous non-damping acoustic media and has been well studied. There exist several methods to solve the corresponding inverse  problem of PAT such as explicit inversion formulas \cite{FPR,XW05,Kun07,FHR,IPI,Halt2d,Halt-Inv,natterer2012photo,Pal-Uniform}, series solutions \cite{kunyansky2007series,agranovsky2007uniqueness}, time reversal \cite{FPR,HKN,HTRE,US,stefanov2011thermo}, and quasi-reversibility \cite{clason2008quasi}. Reviews on these methods can be found in \cite{HKN,kuchment2014radon,kuchment2008mathematics,RosNtzRaz13}.
Discrete iterative approaches, which are based on a discretization of the forward problem
together with numerical solution methods for solving the resulting system of linear equations can be found
in
\cite{paltauf2002iterative,paltauf2007experimental,zhang2009effects,deanben2012accurate,wang2012investigation,rosenthal2013current,huang2013full,wang2014discrete}. Recently, iterative schemes in a Hilbert space settings have also  been introduced and studied; see \cite{arridge2016adjoint,belhachmi2016direct,haltmeier2017iterative}.

\paragraph{PAT in heterogenous damping media:}
In this article, we are interested in  PAT accounting for   spatially variable sound speed and spatially variable damping. It is still an ongoing research which is the correct model for attenuation, and several different modeling equations have been used
(see, or example,  \cite{lariviere2005image,lariviere2006image,ammari2012photoacoustic,kowar2012photoacoustic,acosta2017thermoacoustic,ammari2011time,burgholzer2007compensation,homan2013multi,kowar2014time,palacios2016reconstruction,treeby2010photoacoustic}).   For mathematical interest,
we consider a simple attenuation model using the  damped wave equation, which reads
\begin{equation} \label{E:PAT}
\left. \begin{array}{ll} [c^{-2}(x) \, \partial_{tt} + a(x) \, \partial_{t} -  \Delta] p(x,t) =0  & \mbox{on } \R^d \times \R_+,
\\[6 pt] p(x,0) =f(x)  & \mbox{on }  \R^d , \\[6 pt] 
p_t(x,0) = -c^2(x) \, a(x) f(x) & \mbox{on } \R^d \,.\end{array} \right.
\end{equation}
Here, $c \colon \R^d \to \R$ is the variable sound speed,  $a \colon \R^d \to \R$ the variable damping coefficient, and $f \colon \R^d \to \R$ the desired  initial pressure. We assume that $c$ and $a$ are smooth functions, $c$ is bounded between two positive constants, and $a \geq 0$. Let us denote by $S$ the observation surface and by $T>0$ the final measurement  time. We will assume that $S$ is a (relatively) closed subset of $\partial \Om$ with nonempty interior $\interior(S)$, where $\Om$ is an open subset of $\R^d$ that contains the support of $f$. The mathematical problem of PAT is to invert the map $\Wo\colon f \mapsto g \coloneqq p|_{S \times (0,T)}$. It is referred to as the inverse source problem of PAT. In this article, we assume that $\Wo$ is injective (that is, the reconstruction is unique). For the full data problem, it holds as long as $T> \max_{x \in \Om} \mbox{dist}(x, \partial \Om)$ (see,  \cite{acosta2017thermoacoustic}). The injectivity of $\Wo$ in the case of partial data is still an open problem and beyond the scope of this article.

There are only few papers analyzing the damped wave equation~\eqref{E:PAT}
for PAT \cite{homan2013multi,palacios2016reconstruction,acosta2017thermoacoustic}.
In \cite{homan2013multi}, some interesting microlocal analysis results have been derived for \eqref{E:PAT} and a 
time-reversal framework for image reconstruction has been proposed. This time reversal method is only proved 
to converge (linearly) to the exact solution when the attenuation coefficient is small enough. In the recent work \cite{palacios2016reconstruction}  a modification of the time reversal method has been proposed that converges (linearly) 
to the solution for arbitrarily large attenuation coefficient. A more general model was considered in \cite{acosta2017thermoacoustic}. Let us mention that, in order for the algorithm to converge, both papers assume that the data is measured on a closed surface completely surrounding the object (i.e., full data problem).
Opposed to that, the analysis and algorithms we derive in the present paper apply to the partial data  problem 
as well as the full data problem.

\paragraph{Main contributions:}

In this article, we establish the mathematical foundation of several reconstruction methods for the inverse source problem of PAT with variable sound speed and damping. Namely, we formulate the adjoint operator in the continuous setting using a nonstandard wave equation. We prove the well-posedness of the adjoint equation in a natural setting and its finite speed of propagation. We then propose and analyze various iterative reconstruction algorithms for PAT employing our knowledge of the adjoint operator.  We study both the full and limited data cases. Under the uniqueness and the visibility condition (described in Section \ref{S:Wellposedness}),
our algorithms converge linearly to the  solution, even for the partial data problem. 
The convergence is shown in the  $L^2$-type norm (on image and pre-image space) and the $H^1$-type norm.  
We note that convergences in the $H^1$-type norm have been a common practice in the
inverse source problem of PAT (see for example \cite{homan2013multi,palacios2016reconstruction}).
However, in practice, the image to be recovered may not be in $H^1$.
Therefore, having convergence in the $L^2$-norm is helpful, too.

In   case that  the visibility condition does not hold, the inverse problem of
PAT is severely ill-posed  and regularization methods have to be applied for its solution.
 For that purpose Landweber's, the steepest  descent and the CG method can be applied as well, 
 since they are known to be regularization methods when combined with Morozov's discrepancy principle \cite{engl1996regularization,hanke1995conjugate,kaltenbacher2008iterative}. 
 Additionally, we study generalized Tikhonov regularization \cite{scherzer2009variational}, which  consists in minimizing
the penalized residual functional
$ \Phi(f) = \frac{1}{2} \snorm{ \Wo f -  g}^2  + \la \, \rf(f) $.
Here $\rf \colon \X \to [0, \infty]$ is a convex regularization term and $\la>0$ is the
regularization  parameter.  In particular, we investigate the quadratic, $\rf(f)  = \int_\Omega \abs{ \nabla f}^2$, and the total variation (TV), $\rf(f)  = \int_\Omega \abs{ \nabla f}$, regularizations.
In the quadratic case, the above iterative methods can again be applied to 
minimize  $\Phi$. For the latter case,  we use the minimization algorithm  of \cite{sidky2012convex}, which is a special instance of the  Chambolle-Pock algorithm \cite{chambolle2011}.
Using a discretization of the forward operator with matched discrete adjoint,
variational methods including TV minimization have been applied   in \cite{huang2013full}.
Using continuous formulations of the adjoint, variational methods have been  applied to PAT
in~\cite{arridge2016accelerated,javaherian2017multi}. Our application of variational regularization
for the  damped wave equation~\eqref{E:PAT} is new.

\paragraph{Outline:}
The article is organized as follows. In Section~\ref{S:Adjoint}, we derive the explicit formulation of the adjoint operator. We also discuss some properties of the adjoint equation. In Section \ref{sec:inverse}, we study the inverse problem of PAT in inhomogeneous  damping media. We show that the inverse problem of PAT is well-posed under the visibility condition (see Subsection~\ref{S:Wellposedness}). We analyze iterative  and variational reconstruction algorithms in the well-posed and the ill-posed cases.  In section~\ref{S:Numerical}, we present  various numerical examples for the proposed methods. The main theoretical result, the analysis of the adjoint equation, is presented Appendix~\ref{A:well}. We briefly describe the $k$-wave method, which we use for our forward and adjoint simulation, in Appendix~\ref{app:kspace}. 

\section{The adjoint operator for PAT} \label{S:Adjoint}

Let us recall that the PAT forward operator is given by $\Wo\colon f \mapsto g \coloneqq p|_{S \times (0,T)}$, where $p$ is defined by the acoustic wave equation (\ref{E:PAT})
and $S$ is a closed subset of $\pd \Om$. Our goal is to invert $\Wo$ using the methods introduced in the following  section.
 It is crucial to analyze the adjoint operator $\Wo^*$ of $\Wo$. To that end, we first need to identify the correct mapping spaces
 for $\Wo$. We, indeed, will consider two realizations, $\Wo_0$ and $\Wo_1$, of $\Wo$ corresponding
 to two different choices of the mapping spaces.

We first assume that $\supp(f) \subset \Om_0$, where $\Om_0 \Subset \Om$. For the spaces of $f$, let us denote
\begin{eqnarray*}
\X_0 &\coloneqq \{ f \in L^2(\R^d) \colon \supp(f) \subset \overline{\Om}_0 \}, \\
\X_1 & \coloneqq \{ f \in H^1(\R^d) \colon \supp(f) \subset \overline{\Om}_0 \}.
\end{eqnarray*}
Then, $\X_0$ and $\X_1$ are Hilbert spaces with the respective norms
$\|f\|_{\X_0}  = \|c^{-1} f\|_{L^2(\Om_0)} $ and $\|f\|_{\X_1}  =  \|\nabla f\|_{L^2(\Om_0)}$. We note that $\X_0 \cong L^2(\Om_0)$ and $\X_1 \cong H_0^1(\Om_0)$. The above chosen norms are convenient for our later purposes.

For the spaces of $g$, we fix a nonnegative function $\chi \in C^\infty(\partial \Om \times [0,T])$ such that $\supp(\chi) = \Gamma \coloneqq S \times [0,T]$. Let us denote:
\begin{eqnarray*}
    \Y_0 &=& \left\{g \colon \|g\|_{\Y_0} \coloneqq \|\sqrt{\chi} \, g\|_{L^2(\Gamma)}< \infty \right\}, \\
    \Y_1 &=& \left\{g \colon  g(\edot,0) \equiv 0,~ \|g\|_{\Y_1}  \coloneqq  \|g_t\|_{\Y_0} < \infty \right \}.
 \end{eqnarray*}
We
define
$$ \Wo_i = \Wo|_{\X_i} \colon  (\X_i, \enorm{}_{\X_i}) \to (\Y_i, \enorm{}_{\Y_i})
\quad \mbox{ for }i=0,1. $$
Let $H^i(\Gamma)$ be the standard Sobolev space of order $i$ on $\Gamma$. Notice that $\Wo$ is a bounded map from $\X_i \to H^i(\Gamma)$. This comes from the fact that $\Wo$ is the sum of two Fourier integral operators of oder zero (see, e.g., \cite[Lemma~3]{homan2013multi}). Since $H^i(\Gamma) \subset \Y_i$, we obtain:

\begin{theorem} For $i=0,1$, $\Wo_i$ is a bounded map from $\X_i$ to $\Y_i$.
\end{theorem}

From now one, we consider $\chi \, g$ as a function on $\pd \Om \times [0,T]$, which vanishes on $(\pd \Om \setminus S) \times [0,T]$. The following theorem gives us an explicit formulation of the adjoint operator $\Wo_i^*$ of $\Wo_i$:

\begin{theorem} \label{T:adjoint} The following results hold.
\begin{enumerate} [leftmargin=4em,label=(\alph*)]
\item Let $g \in H^1([0,T]; H^{-1/2}(\partial \Om)) \cap \X_0$. Consider the wave equation
\begin{eqnarray} \label{E:adjoint}
\left.\begin{array}{ll} [c^{-2} \, \partial_{tt} - a \partial_t - \Delta] q =0,
&  (\R^d \setminus \partial \Om) \times (0,T), \\[6 pt] q(T) =0,~ q_t(T) =0,  \\[6 pt]
\big[ q \big] =0, \Big[\frac{\partial q}{\partial \nu} \Big] =\chi  \, g. \end{array} \right.
\end{eqnarray}
Here, $[q]$ denote the jump of $[q]$ across the boundary $\partial \Omega$. Then $$\Wo^*_0 g = q_t(0)|_{\Om_0}.$$

\item  Let $g \in H^1([0,T]; H^{-1/2}(\partial \Om)) \cap \X_1$. Assume further that $\chi$ is independent of $t$ (i.e., $\chi(y,t) =\chi(y)$). We define
\[\bar g(x,t) = g(x,t) - g(x,T),\]
and consider the wave equation
\begin{eqnarray} \label{E:adjoints}
\left.\begin{array}{ll} [c^{-2} \, \partial_{tt} - a \, \partial_t - \, \Delta] \bar q =0,
& (\R^d \setminus \partial \Om) \times (0,T), \\[6 pt] \bar q(T) =0, \quad  \bar q_t(T) =0,  \\[6 pt]
\big[ \bar q \big] =0, \Big[\frac{\partial \bar q}{\partial \nu} \Big] =\chi  \, \bar g. \end{array} \right.
\end{eqnarray}
Then,
$$\Wo^*_1 g = \Pi [\bar q_t(0)].$$
Here, $\Pi$ is the projection on the space $\X_1 \cong H_0^1(\Om_0)$, given by
$$\Pi(f) = f -\phi,$$ where $\phi$ is the harmonic extension of $f|_{\partial \Om_0}$ to $\overline \Om_0$.
\end{enumerate}
\end{theorem}

The proof for Theorem~\ref{T:adjoint} is similar to that of \cite[Theorem~3.2]{haltmeier2017iterative}. We skip it for the sake of brevity. The analysis of (\ref{E:adjoint}), which is the main theoretical achievement of this article, is presented in Theorem~\ref{E:well-posed}. Namely, we show that if $g \in  H^1([0,T];H^{-1/2}(\partial \Om))$, equation (\ref{E:adjoint}) has a unique solution $q \in L^2([0,T]; H^1(\R^d))$ satisfying $q' \in L^2([0,T];L^2(\R^d))$, and $q'' \in L^2([0,T];H^{-1}(\R^d))$. Moreover,  $q$ satisfies the finite speed of propagation: let $c_+ = \max_{x \in \R^d} c(x)$, then $q(x,t) =0$ for any $(x,t) \in \Om^c \in [0,T]$ such that $dist(x,\partial \Om) \geq c_+ (T-t)$.  In the absence of damping (i.e., $a=0$), an existence and uniqueness of equation (\ref{E:adjoint}) has been proved in \cite{belhachmi2016direct}.  Compared to their result, we require less regularity on $g$ and the solution space is more natural. Moreover, the finite speed of propagation is new. It helps us to truncate the calculation domain when needed.

\begin{remark}
Let us make the following observations:
\begin{enumerate}[leftmargin=4em,label=(\alph*)]
\item Since $H^1([0,T];H^{-1/2}(\partial \Om)) \cap \Y_i$ is dense in both $\Y_i$ for $i=0,1$, the adjoint operators $\Wo_0^*$ and $\Wo_1^*$ are uniquely determined from the formulas in Theorem~\ref{T:adjoint}.

\item Compared to $\Wo_0^*$,  $\Wo_1^*$ involves an extra projection operator. In our numerical experiments, we will only use $\Wo_0^*$ since it is simpler to implement. However, the knowledge of $\Wo_1^*$ is helpful in designing iterative algorithms that converge in the $H^1$-norm.

\end{enumerate}
\end{remark}

\section{Solution to the inverse problem} \label{sec:inverse}

In this section, we present  methods for inverting the  two realizations  $\Wo_i \colon \X_i \to \Y_i$ for $i  =0,1$.
To that end, we first show that the inverse problems are  well
posed under the visibility condition. We then separately
consider the well-posed and ill-posed situation.

\subsection{Well-posedness under the visibility condition}
\label{S:Wellposedness}

Let us fix several geometric conventions. We will always assume that the sound speed $c$ is smooth and bounded from below by a positive constant. The space $\R^d$ is considered as a Riemannian manifold with the metric $c^{-2}(x) \, dx^2$ and $\Om$ is assumed to be strictly convex with respect to this metric. Then, all the geodesic rays originating inside $\Om$ intersect the boundary $\pd \Om$ at most once. We also assume that the speed $c$ is nontrapping, i.e.,  all such geodesic rays intersect with $\pd \Om$. Also, $\cT^* \Om \setminus 0$ is the cotangent bundle of $\Om$ minus the zero section, which can be identified with $\Om \times (\R^d \setminus \set{0})$.

{\bf Visibility condition:} \emph{There is a closed subset $S_0 \subset \partial \Om$ such that $S_0 \subset \interior(S)$ and the following condition holds: for any element $(x,\xi) \in \cT^*\Om_0 \setminus 0$, one of the unit speed geodesic rays originating from $x$ at time $t=0$ along the directions $\pm \xi$ intersects transversally with $S_0$,  at a time $t<T$.}

Let us recall that, in this article, we will always assume the injectivity of $\Wo_i$. Our first result is that the inversion of $\Wo_i$ is stable under the visibility condition.

\begin{theorem} \label{T:well-posed} Assume that the visible condition holds and $\chi \equiv 1$ on $S_1 \times [0,T]$, where $S_1$ is a closed subset of $\partial \Om$ such that $S_0 \subset \interior(S_1)$ and $S_1  \subset \interior(S)$.
For $i=0,1$, there is a constant $C>0$ such that for any $f \in \X_i$, we have
\begin{equation}\label{eq:wellposed}
\|f\|_{\X_i} \leq C \|g\|_{\Y_i} \quad  \mbox { where }  g= \Wo  f \,.
\end{equation}
\end{theorem}
One proof virtually follows from \cite[Theorem~3.4]{haltmeier2017iterative} line by line. One only needs to refer to \cite{homan2013multi} instead of \cite{US} when needed. We briefly present here another approach. 
\begin{proof}
Observe that $\Wo_i^* \Wo_i$ is, similarly to the non-damping case (see \cite[Theorem~3.6]{haltmeier2017iterative}), an elliptic operator from $\X_i$ into itself with the principal symbol $\sigma(x,\xi)$ being bounded from below by a positive constant $\delta$. We then have
$$\|\Wo_i f\|_{\Y_i}^2 = \left< \Wo_i^* \Wo_i f, f \right>_{\X_i} \geq \delta \left< f, f \right>_{\X_i} + \left<\Ko f, f \right>_{\X_i}, $$ where $\Ko$ is a compact operator. 
Young's inequality gives
$$\|f\|^2_{\X_i} \leq C(\|\Wo_i f\|_{\Y_i}^2  + \|\Ko f\|_{\X_i}).$$
The injectivity of $\Wo_i$ and \cite[Theorem V.3.1]{taylor1981pseudodifferential} gives
$$\|f\|^2_{\X_i} \leq C \|\Wo_i f\|_{\Y_i}^2.\qedhere$$
\end{proof}

\begin{algorithm}
\begin{algorithmic}[1]
\STATE Initialize $f_0^\delta=0$; $k \gets 0$
\WHILE{stopping criteria not satisfied}
\STATE $ s_k = \Wo_i^* ( \Wo_i f^\delta_k - g^\delta)$
\STATE $ \gamma_k = \snorm{s_k}_{\X_i}^2 / \norm{\Wo_i s_k}_{\Y_i}^2$
\STATE $f^\delta_{k+1} = f^\delta_k - \gamma_k s_k$
\STATE $k  \gets  k+1$
\ENDWHILE
\end{algorithmic}
\caption{Steepest\label{alg:sd} descent method for $\Wo_i f  =  g^\delta$.}
\end{algorithm}

\begin{algorithm}
\begin{algorithmic}[1]
\STATE Initialize $f^\delta_0=0$; $r_0= g^\delta - \Wo_i f_0^\delta$; $d_0=\Wo_i^* r_0$; $k \gets 0$
\WHILE{stopping criteria not satisfied}
\STATE $\alpha_k =\| \Wo_i^* r_k\|_{\X_i}^2/ \|\Wo_i d_k\|_{\Y_i}^2$
\STATE $f^\delta_{k+1} = f^\delta_k  + \alpha_k \, d_k$
\STATE $r_{k+1} =  r_k -  \alpha_k \, \Wo_i d_k$
\STATE $\beta_k = \|\Wo_i^* r_{k+1} \|_{\X_i}^2/ \|\Wo_i^* r_k\|_{\X_i}^2$
\STATE $d_{k+1} = \Wo_i^* r_{k+1} + \beta_k \, d_k$
\STATE $k  \gets  k+1$
\ENDWHILE
\end{algorithmic}
	\caption{CGNE\label{alg:cgne} method for $\Wo_i f  =  g^\delta$.}
\end{algorithm}

\subsection{Well posed case: Linear convergence of iterative methods}

When the  linear inverse problem $\Wo f  =  g$ is well-posed,
then Landweber's,  the steepest descent, and the CG methods applied to
$g^\delta$ converge  to a minimizer of
\begin{equation} \label{eq:res}
	\Phi_{0}  \colon \X_i \to \R \colon  f
	\mapsto \frac{1}{2} \snorm{\Wo_i f - g^\delta}_{\Y_i}^2 	
\end{equation}
with a linear rate of convergence (for both realizations $\Wo_i \colon \X_i \to \Y_i$ of  $\Wo$).
Here, we assume that $f \in \dom(\Wo_i)$ and $g^\delta \in \Y_i$ is such that $\norm{\Wo_i f - g^\delta}_{\Y_i} < \delta$.
For convenience of the reader the steepest  descent and the
CG iteration are  recalled  in  Algorithms~\ref{alg:sd} \&
\ref{alg:cgne}. The Landweber's method is the same as the
steepest descent method with the modification that  the step size $\gamma_k$ is 
replaced by a constant value $\gamma$ satisfying  $0< \gamma < 2 / \snorm{\Wo_i^*\Wo_i}$.
Theorem~\ref{T:well-posed} implies the following  result.

\begin{theorem} \label{thm:iter}
Assume that the visible condition holds and let $\chi \equiv 1$ on $S_1 \times [0,T]$,
where $S_1$ is a closed subset of $\partial \Om$ such that $S_0 \subset \interior(S_1)$ and $S_1  \subset \interior(S)$.
\begin{itemize}[topsep=0em]
\item
For any $g^\delta \in \Y_i$, the  Landweber, the steepest  descent and the CG iteration
converge linearly to the unique minimizer $f^\delta$ of \eqref{eq:res}. More precisely, there is a
constant $a < 1$ (only depending on the realization  and the iterative method) such that the iterates $f_k^\delta$  defined  by either method satisfy
$\snorm{f^\delta  -  f^\delta_k }_{\X_i}  \leq a^k \snorm{f^\delta}_{\X_i}$ for
$k \in \N$.

\item
For $\delta = 0$, the limit $f^0$ is the unique solution of $\Wo_i f  =  g$.
Moreover, we have $\snorm{f - f^\delta}_{\X_i} \leq  C \delta$, where $C$ is the constant appearing in
Theorem~\ref{T:well-posed}.
 \end{itemize}
\end{theorem}

\begin{proof}
Theorem~\ref{T:well-posed} shows that the inverse problem is well-posed. The above results follow directly from the standard theory of iterative methods \cite{hanke1995conjugate,engl1996regularization,kaltenbacher2008iterative}.
\end{proof}

Theorem~\ref{thm:iter} shows that with our choices of mapping spaces, the Landweber's,  steepest descent, and 
CG methods converge linearly in the $L^2$-norm as well as the $H^1$-norm.

\subsection{Ill-posed case: regularization}

Now consider the  situation where the visibility  condition   does not hold.
Then one has to apply regularization methods.

\paragraph{Iterative regularization methods:}

We  consider the Landweber, the steepest descent and the CG methods
combined with Morozov's discrepancy principle.
According to the discrepancy principle,
the iteration is terminated at the index  $$k(\delta, g^\delta) = \argmin \set{ k  \in \N \colon  \snorm{ \Wo_i f_{k+1}^\delta  -  g^\delta }_{\X_i}  \leq \tau \delta}$$ with some fixed $\tau >1$.

\begin{theorem}
Suppose $f \in   \X_i$, $\delta>0$, let  $g^\delta \in \Y_i$
satisfy $\snorm{g^\delta - \Wo  f}_{\Y_i}  \leq \delta$ and
define $(f_k^\delta)_{k\in \N}$ by either the Landweber,
steepest  descent or the CG iteration.

\begin{enumerate}
\item \emph{Exact data:} If $\delta=0$, then
$ \norm{ f_k-f}_{\X_i}  \to 0 $ as $k \to \infty$.

\item \emph{Noisy data:}
Let $(\delta(m))_{m\in \N} \in (0, \infty)^\N$ converge  to zero and let
$(g_m)_{m\in \N} \in \Y_i$
satisfy $\snorm{g_m- \Wo f}_{\Y_i} \leq \delta(m)$.
Then the following hold:
\begin{itemize}
\item The stopping indices $k_*(\delta(m), g_m) $  are  well defined;
 \item
 We have  $\snorm{f_{k_*(\delta(m), g_m)}^{\delta(m)} - f}_{\X_i} \to 0 $ as $m\to \infty$.
\end{itemize}
\end{enumerate}
\end{theorem}

\begin{proof}
The claims follow  from standard  results for iterative regularization methods
(see, for example, \cite{hanke1995conjugate,engl1996regularization,kaltenbacher2008iterative}).
\end{proof}

\paragraph{Variational (penalized) regularization methods:}

As an alternative  to iterative regularization methods we will apply generalized
Tikhonov regularization, which has the advantage  that  a-priori information
can be more easily explicitly incorporated.
In this work  we apply $H^1$-regularization  and TV-regularization,
\begin{align} \label{eq:H1}
	\Phi_2 (f)   &\coloneqq \frac{1}{2} \snorm{ \Wo  f -  g}^2_{\Y_0}  + \frac{\la}{2}  \,  \int_{\Om_0}  \abs{\nabla f}^2 \,,
	\\  \label{eq:TV}
	\Phi_1 (f)  &\coloneqq \frac{1}{2} \snorm{ \Wo  f -  g}^2_{\Y_0}  + \la \, \int_{ \Om_0} \abs{\nabla f}  \,,
\end{align}
respectively. Here $\la  > 0$ is the regularization  parameter and both functionals are considered as
mappings on $\X_0 = L^2 (\Omega_0)$. From the general theory of variational regularization
methods,     it follows  that \eqref{eq:H1} and \eqref{eq:TV}  again
yield regularization methods \cite{scherzer2009variational}.

For numerically minimizing the Tikhonov functionals  \eqref{eq:H1}
and \eqref{eq:TV},  we replace them by the discrete counterparts
    \begin{align} \label{eq:H1n}
	\Tnum_2 ( \fnum)
	&\coloneqq\frac{1}{2}  \snorm{ \Wnum  \fnum -  \gnum^\delta}_2^2  +
	\frac{\la}{2} \, \norm{  | \Dnum \fnum | }_2^2  \,,
	\\  \label{eq:TVn}
	\Tnum_1 (\fnum)
	&\coloneqq \frac{1}{2} \snorm{ \Wnum  \fnum -  \gnum^\delta}_2^2
	+ \la \, \norm{  | \Dnum \fnum | }_1 \,.
\end{align}
Here $\fnum \in \R^N$, $\gnum^\delta \in \R^M$,
$\Wnum \colon \R^N \to \R^M$ is the discretization of the forward operator and
  $\Dnum \colon \R^N \to \R^{N}  \times \R^{N}$ denotes the discrete gradient.
  The functional \eqref{eq:H1n} is quadratic and can be minimized, for example, with the  steepest  descent or the CG iteration.
The discrete  TV problem~\eqref{eq:TVn} can also be minimized  by  various methods.
In this work we use the minimization algorithm  of \cite{sidky2012convex}, which is a special instance of the  Chambolle-Pock algorithm \cite{chambolle2011} and summarized in  Algorithm~\ref{alg:tv}.

\begin{algorithm}
\caption{\label{alg:tv} Algorithm for minimizing \eqref{eq:TVn}}
\label{alglstv}
\begin{algorithmic}[1]
\STATE
$L \gets \|( \Wnum,  \Dnum)\|_2$;
$\tau \coloneqq 1/L$;
$\sigma \coloneqq 1/L$;
$\theta \coloneqq 1$;
$ k \gets 0$
\STATE initialize $\fnum_0$, $\pnum_0$, and $\qnum_0$ to zero values
\STATE $\unum_0 \gets \fnum_0$
\WHILE{stopping criteria not satisfied}
\STATE $\pnum_{k+1} \gets (\pnum_k + \sigma( \Wnum \unum_k - \gnum^\delta))/(1+\sigma)$
\STATE $\qnum_{k+1} \gets \lambda (\qnum_k  + \sigma \Dnum \unum_k )/
\max \set{ \lambda \mathbf{1},|\qnum_k + \sigma \Dnum \unum_k |}$
\STATE $\fnum_{k+1} \gets  \fnum_k - \tau \Wnum^\trans \pnum_{k+1} + \tau \Dnum^\trans  \, \qnum_{k+1}$
\STATE $\unum_{k+1} \gets \fnum_{k+1} + \theta(\fnum_{k+1} - \fnum_k)$
\STATE $k \gets k+1$
\ENDWHILE
\end{algorithmic}
\end{algorithm}

\section{Numerical examples}
\label{S:Numerical}

In this section we present numerical examples for full   data  (well-posed case)
as well as for limited view data (ill-posed case).
For both cases we take $\Om = [-1,1]^2$  and $\Om_0 = B_{0.9}(0)$, the ball with radius 0.9 centered at the origin. We also assume variable sound speed  and  variable attenuation
profile. We  consider the realization of the operator
$\Wo = \Wo_0 \colon \X_0 \to \Y_0$ using the $L^2$-norm.
For the forward and the adjoint equations, the wave equation
is solved with  a variant of the $k$-space method that is described in
Appendix~\ref{app:kspace}. The $k$-pace method yields solutions
that are periodic with period determined by the size of  the computational
domain. To avoid effects of periodization in all numerical  simulations
the domain $\Om = [-1,1]^2$ is embedded in a larger computational domain
$[-2,2]^2$.

 \begin{figure}[tbh!]\centering
\includegraphics[width =\columnwidth]{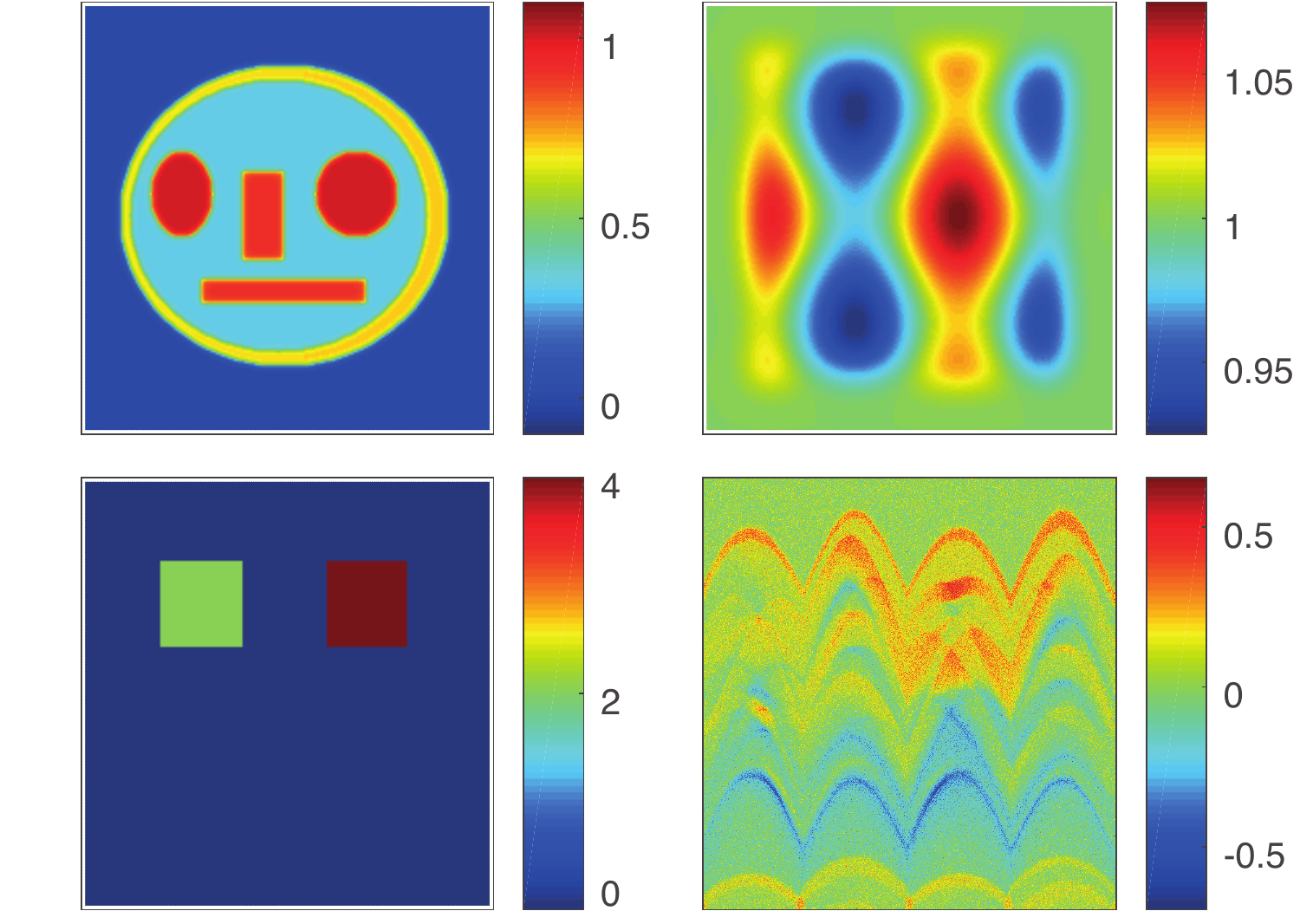}
\caption{Phantom (top left), variable sound speed (top right), variable attenuation coefficient (bottom left) and data with added noise (bottom right).\label{fig:data}}
\end{figure}

The initial phantom, the  sound
speed and the  attenuation are shown in
Figure~\ref{fig:data}.  All these functions are represented
by discrete vectors in $ \R^{201 \times 201}$.
The  computed  data  $g \in  \R^{800 \times 501} $
corresponds to  discrete pressure values at the 800 boundary pixels on
$ \partial \Om $ and   $501$ equidistant time samples  in $[0,2.5]$.
The (full data) discrete forward operator    $\Wnum \colon \R^{201 \times 201}  \to  \R^{800 \times 501}$
is obtained  by restricting   the numerical  solution to
the boundary  pixels. The discretization $\Wnum^\trans \colon    \R^{800 \times 501} \to \R^{201 \times 201}$
of the  adjoint operator
is also computed using the $k$-space method.  In  order  to avoid inverse crime,
in all simulations we use a twice finer  discretization for the data simulation than for the  reconstruction (followed by restriction  to the $800 \times 501$ grid).

\subsection{Full view data (well-posed case)}

We first  study the well-posed case where the  data is given on the whole boundary. The standard iterative methods (Landweber, steepest descent and CG)
 are therefore  linearly convergent.

 \begin{figure}[tbh!]\centering
\includegraphics[width =0.48\columnwidth]{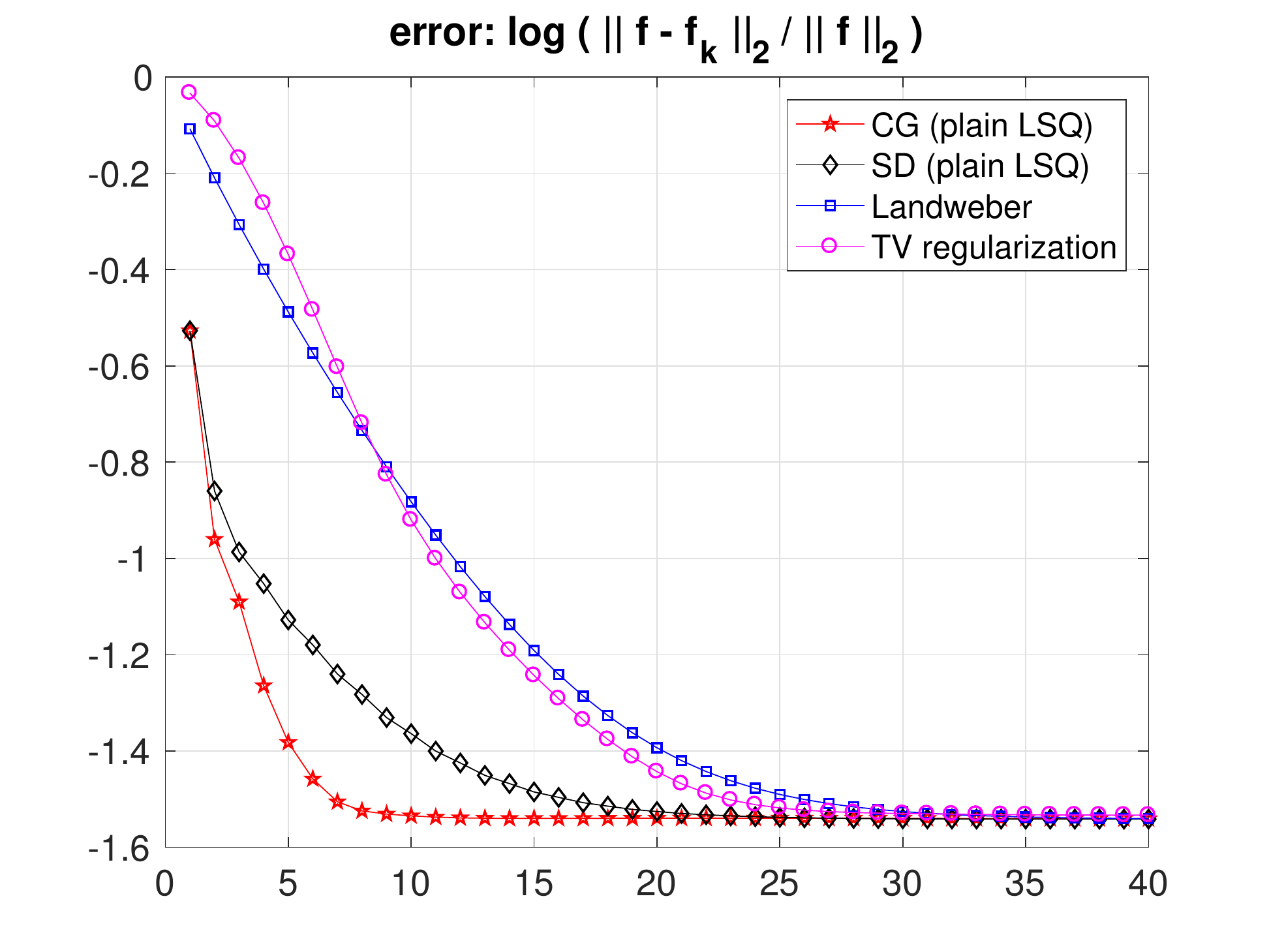}
\includegraphics[width =0.48\columnwidth]{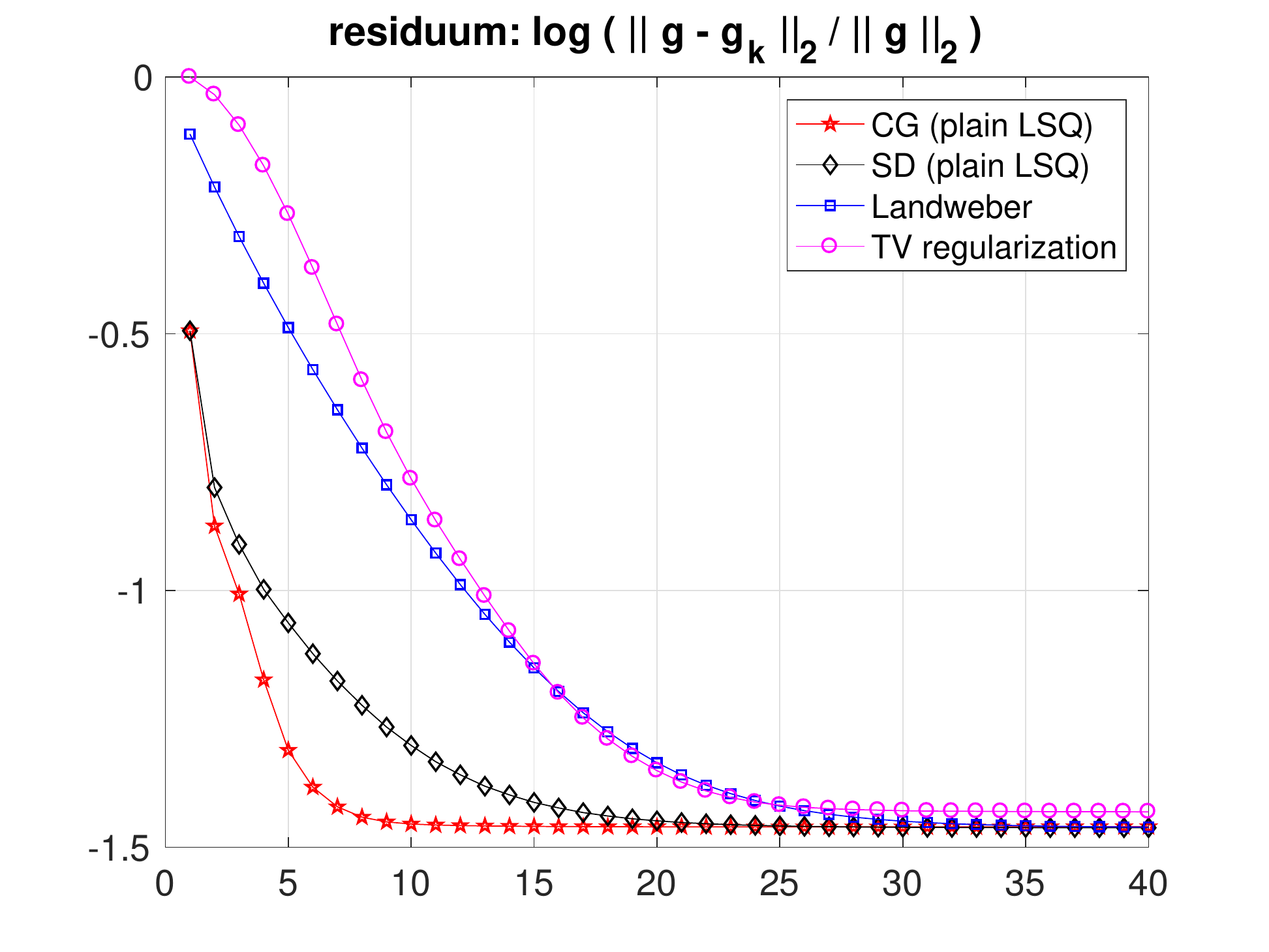}
\caption{Errors and residuals for the full data case  without added noise.\label{fig:err1}}
\end{figure}

\begin{figure}[tbh!]\centering
\includegraphics[width =\columnwidth]{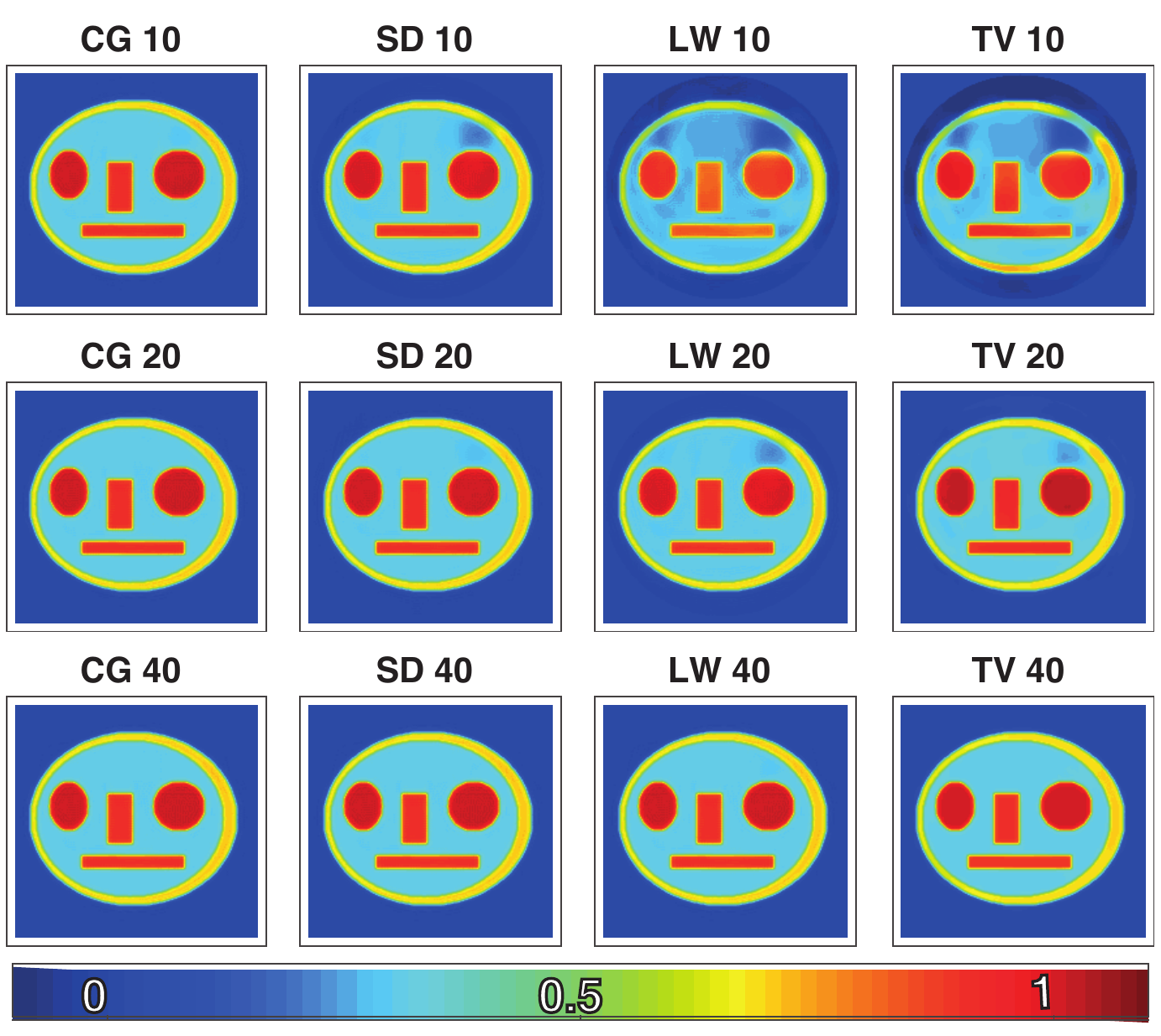}
\caption{Reconstructions after 10, 20 and 40 iterations
for the full data case  without added noise.\label{fig:rec1}}
\end{figure}

 \paragraph{Exact data:}
 
 Figure \ref{fig:err1} shows the   residuals and the
 relative $L^2$-reconstruction errors $\snorm{\fnum_k -\fnum}_2 / \snorm{\fnum}_2$  of the above  methods for the first $40$ iterates applied to simulated data.
  For comparison purpose, we also show results using the TV minimization
  algorithm with $\lambda =0.1$.   One observes that the error and the  residuals stagnate for all methods at some positive value after a certain number of iteration. This is because  the minimizer of
 $\snorm{ \Wnum  \fnum -  \gnum^\delta}_2^2$ is slightly different from the  exact solution $\fnum$ (since $g \neq g^\delta$, mainly due to the different data generation meshes).
 The CG method  is the fastest converging and  the  Landweber the slowest.
  In Figure \ref{fig:rec1}, we show  reconstructions of these methods after
 10, 20 and 40 iterations. All iterative methods have a similar behavior. In the initial iterations there are still artifacts contained in the pictures, and in later iterations the region with high attenuation value is underestimated. After more iterations, also this region is recovered correctly as well. The minimal reconstruction  error $\snorm{\fnum_k -\fnum}_2 / \snorm{\fnum}_2$  is  about $2.9\%$ and the minimal relative  residual $\snorm{\Wnum \fnum_k -\gnum}_2 / \snorm{\gnum}_2$ about $3.5\%$  for all methods.

 \begin{figure}[tbh!]\centering
\includegraphics[width =0.48\textwidth]{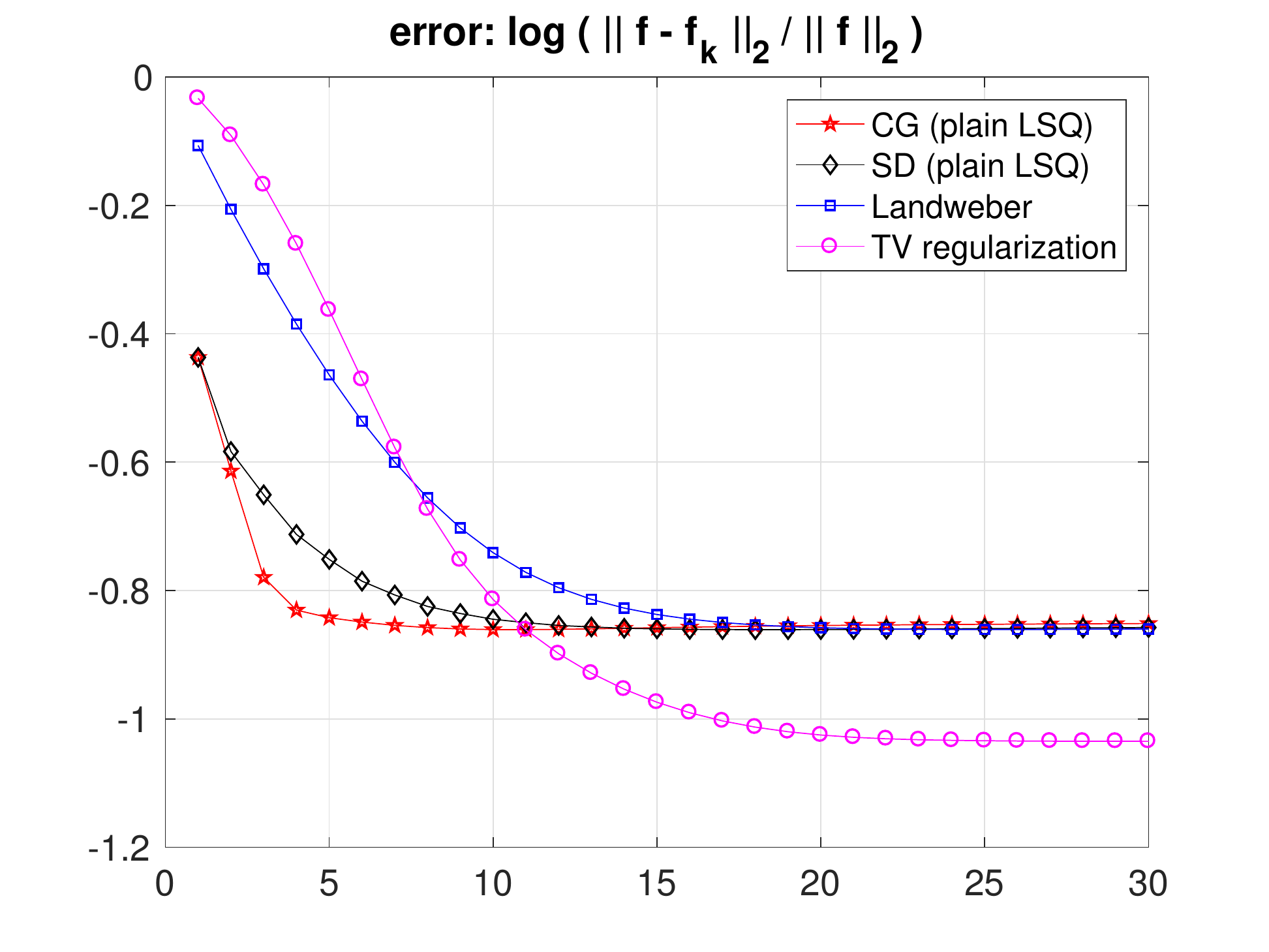}
\includegraphics[width =0.48\textwidth]{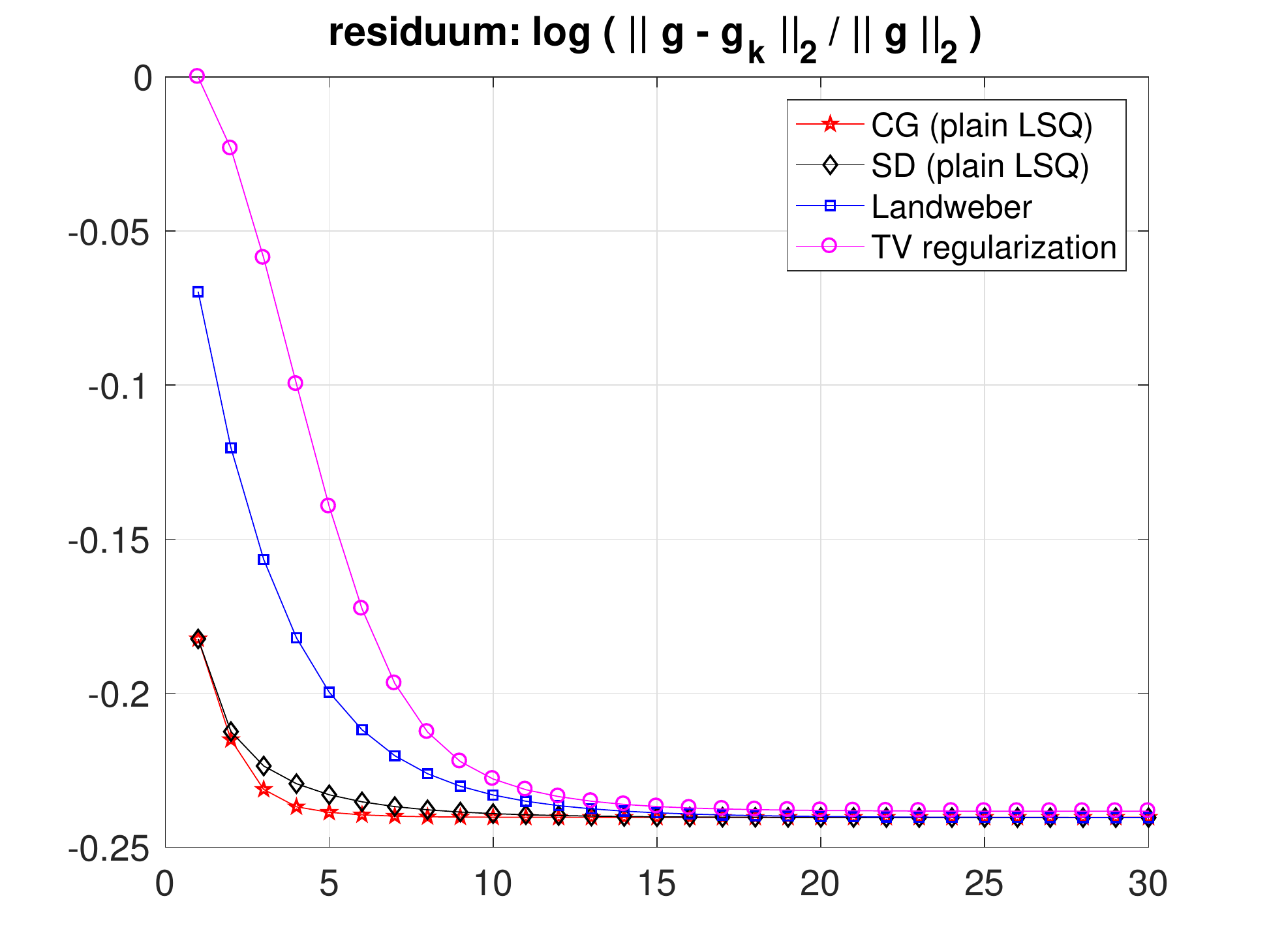}
\caption{Errors and residuals for the full data case with noise added.\label{fig:err1n}}
\end{figure}

\begin{figure}[tbh!]\centering
\includegraphics[width =\columnwidth]{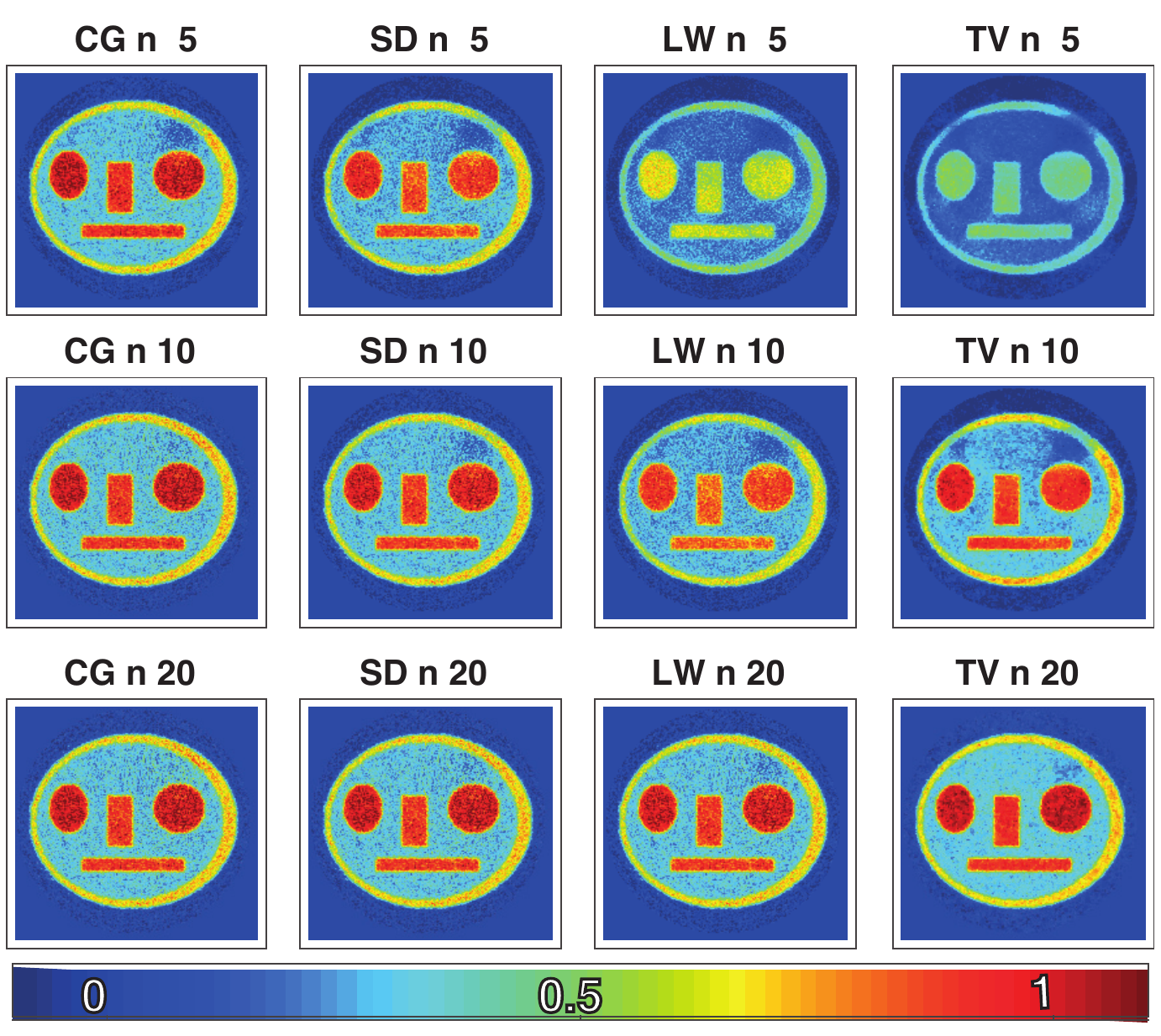}
\caption{Reconstructions after 5, 10 and 20 iterations for the full data case with noise added.\label{fig:rec1n}}
\end{figure}

\paragraph{Noisy data:}

In order to test  stability  with respect to noise  we repeated
the above simulations after  adding uniformly distributed Gaussian  noise
to the data with a relative error of about $59\%$. As can be seen from Figure~\ref{fig:err1n}, the convergence behavior is very similar to the exact data
case reflecting the well-posedness of the inverse problem.
Due to the added noise, the  minimal residuals and the minimal reconstruction errors are of
course much larger than in  exact data case. Reconstructions after 5, 10 and 20 iterations are shown in
Figure~\ref{fig:rec1n}. One observes good reconstruction results and robustness with respect the the noise.
The relative reconstruction errors after 20 iterations are   
about $14\%$, $13.8\%$, $13.9 \%$, $9.4 \%$ for CG, steepest decent, Landweber and 
TV minimization, respectively.    The the relative residuals are 
$57.5\%$,  $57.5\%$,   $57.6 \%$,    $57.88 \%$ which is about the relative data error.
One notes that  the relative reconstruction error is even smaller than the relative data error. This is probably  due to the redundancy of the PAT data. We conclude that in the full data case all methods have similar  stability and accuracy, but the CG is the fastest. Therefore in the case of full data we can suggest the CG method 
among the unpenalized iterative methods for image reconstruction. In the case of the piecewise constant  phantoms TV minimization seems to give better 
results in terms of $L^2$-reconstruction error.

\subsection{Limited view data (ill-posed case)}

Next we consider the limited data where the data are only given on the part of the boundary $\partial [-1,1]^2 $ determined by horizontal component being greater than $-0.25$. The visibility condition is  not satisfied and we are facing a severely ill-posed problem for which one requires a regularization method.
We  propose the steepest descent and CG method as iterative regularization methods and  $H^1$-regularization and TV-regularization as variational regularization methods.
For minimizing the $H^1$-functional  \eqref{eq:H1n} we use the steepest descent iteration which, in our simulations, turned out to be faster than the Landweber method and
more stable than the CG algorithm. For minimizing the TV-functional  \eqref{eq:TVn} we use the  minimization algorithm of \cite{sidky2012convex}. The regularization parameter in the variational methods is set to
$\lambda =0.1$.

\begin{figure}[tbh!]\centering
\includegraphics[width =0.48\textwidth]{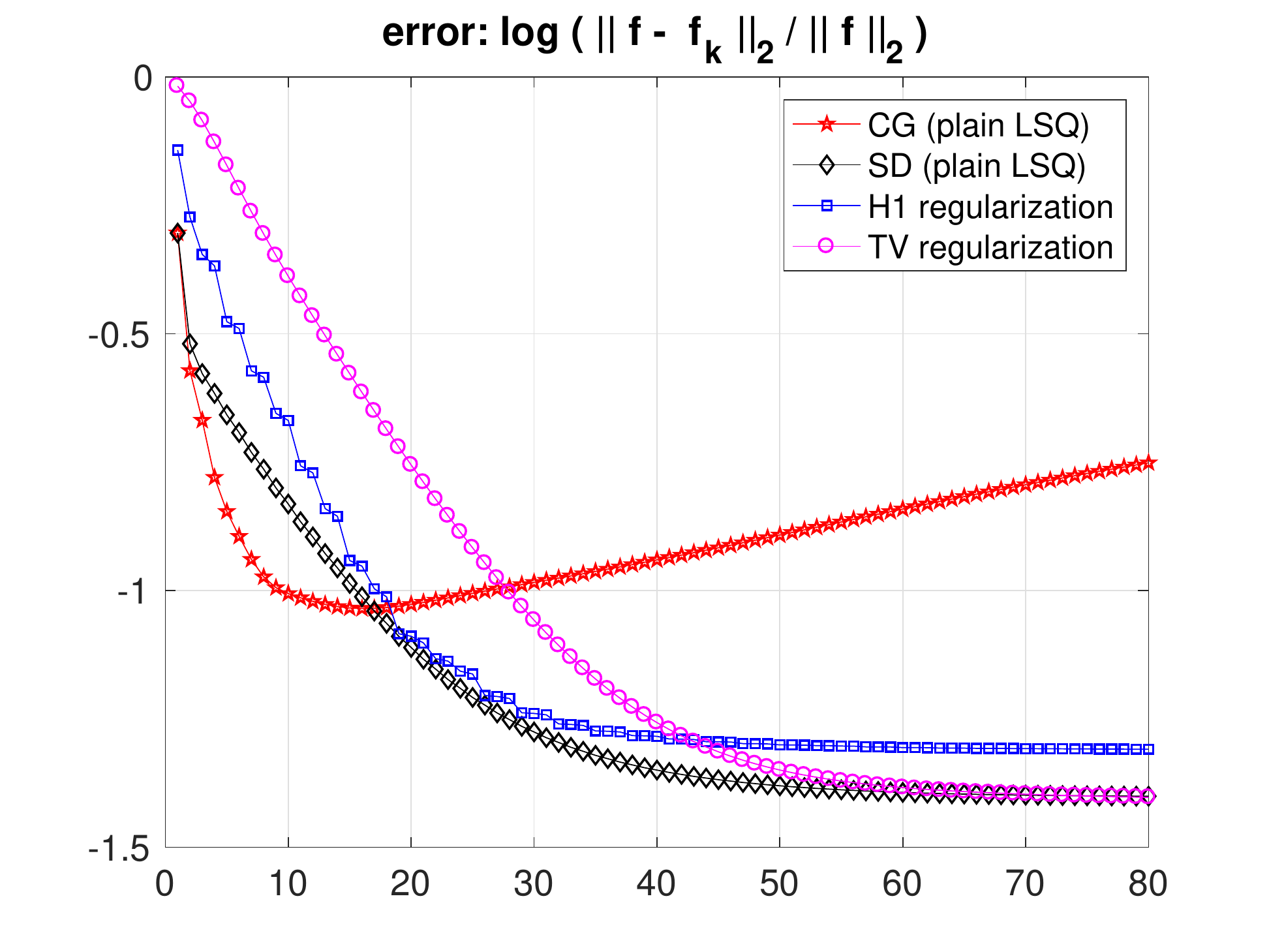}
\includegraphics[width =0.48\textwidth]{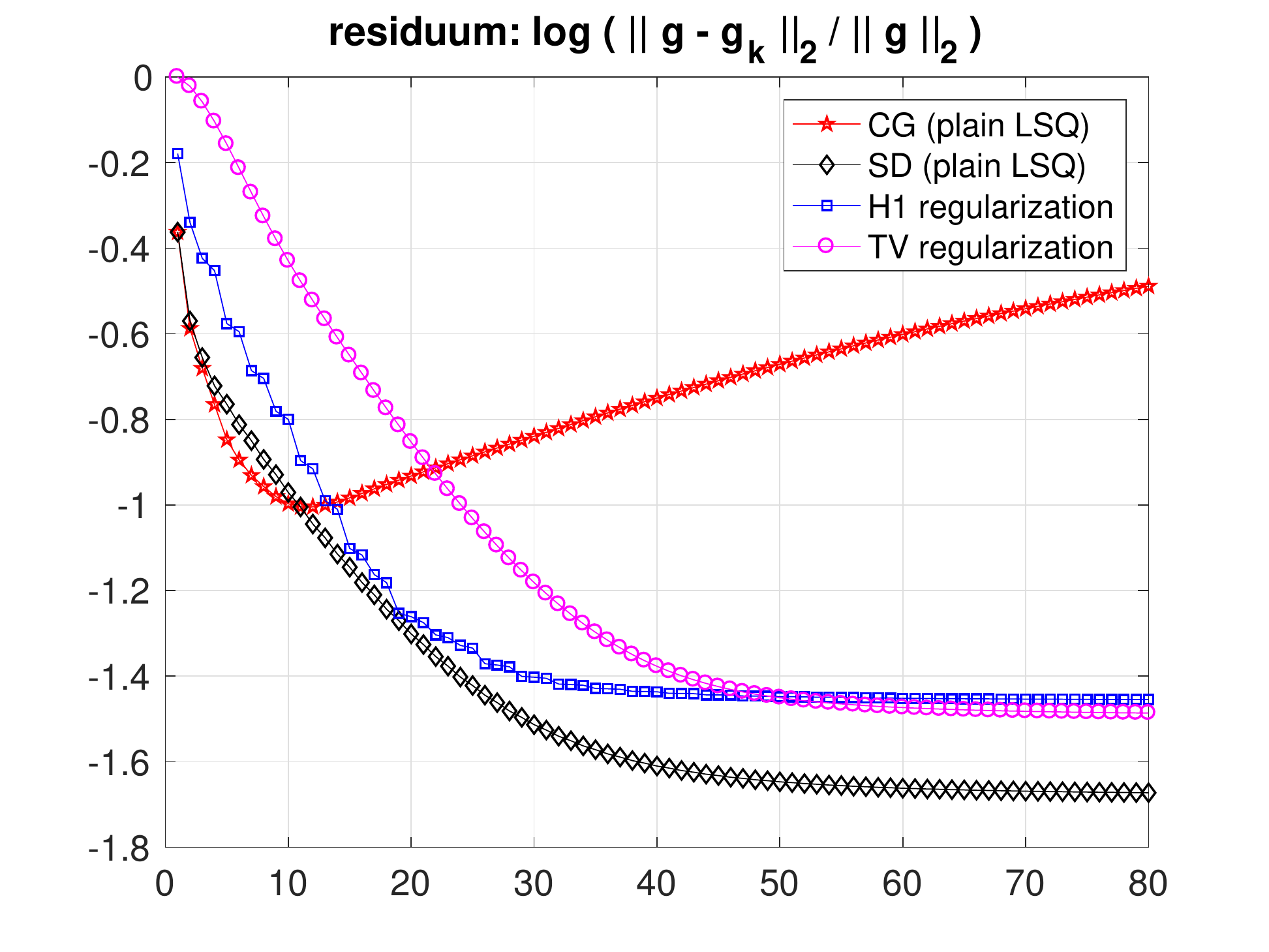}
\caption{Errors and residuals for the ill-posed partial data  case without noise. \label{fig:err2}}
\end{figure}

\begin{figure}[tbh!]\centering
\includegraphics[width=\columnwidth]{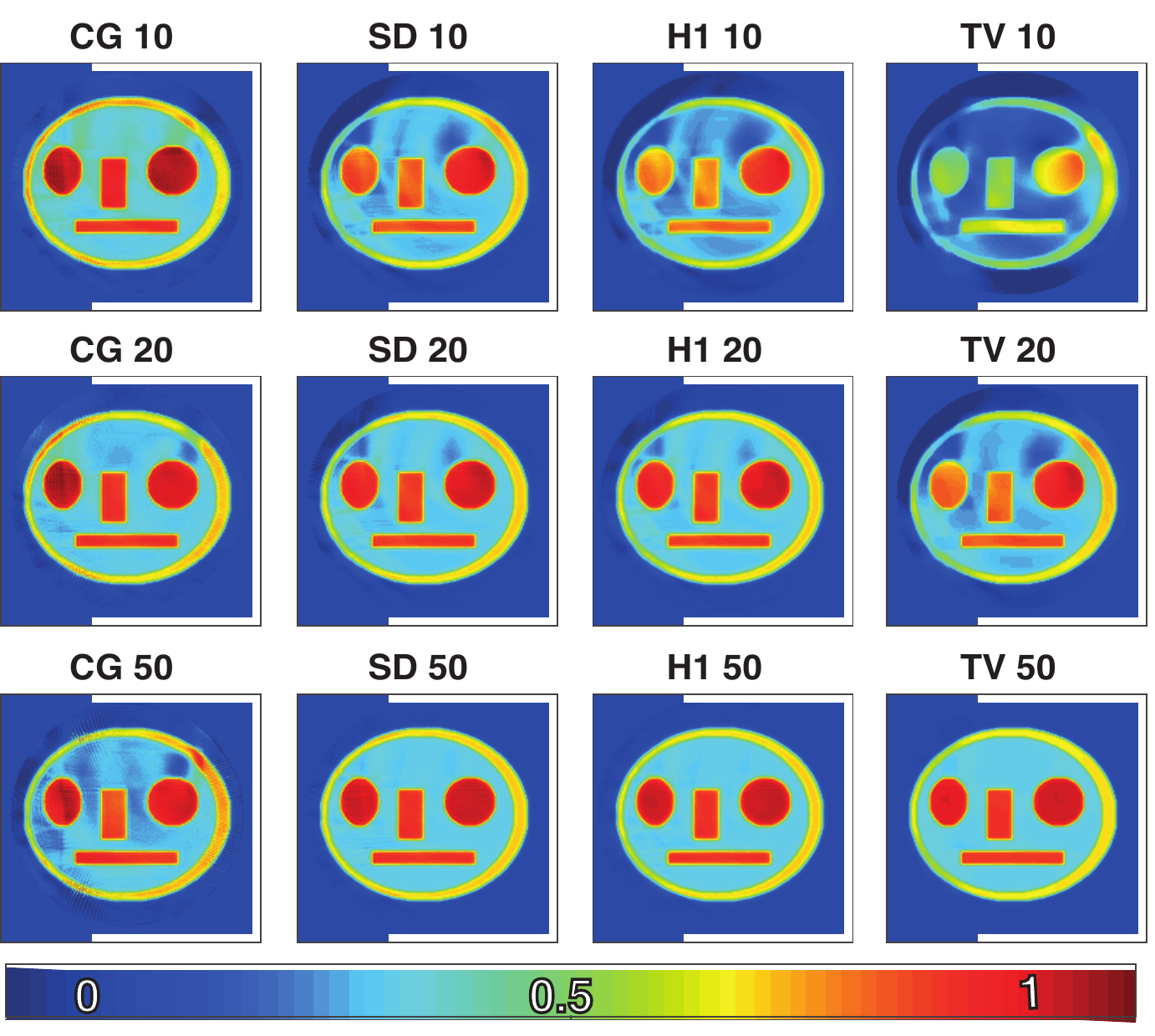}
\caption{Reconstructions after 10, 20 and 50 iterations for the ill-posed partial data  case without noise.\label{fig:rec2}}
\end{figure}

\paragraph{Exact data:}

We start by applying the above  schemes to  the simulated data. Figure \ref{fig:err2} shows the relative errors and relative residuals  for all 
methods on a logarithmic scale. In terms of relative reconstruction 
errors, the steepest descent and the TV algorithm perform best, whereby 
the steepest descent is faster  converging. Surprisingly, while the CG method 
again shows very rapid convergence in the initial iterations, it turns out to be unstable  in the ill-posed case.  Reconstruction results after $10$, $20$ and $50$ 
iterations are shown in the Figure~\ref{fig:rec2}. The relative $\ell^2$-reconstruction  error  after  50 iterations  for the CG iteration, the steepest descent iteration, $H^1$-regularization and TV-regularization are  are $12.8 \%$,   $4.2 \%$,    $5 \%$,  and $4.5 \%$, respectively. The corresponding (relative) residuals are  $21.3\%$,    $2.3 \%$,    $3.6 \%$,  and  $3.6 \%$.

    
    
\begin{figure}[tbh!]\centering
\includegraphics[width=0.48\textwidth]{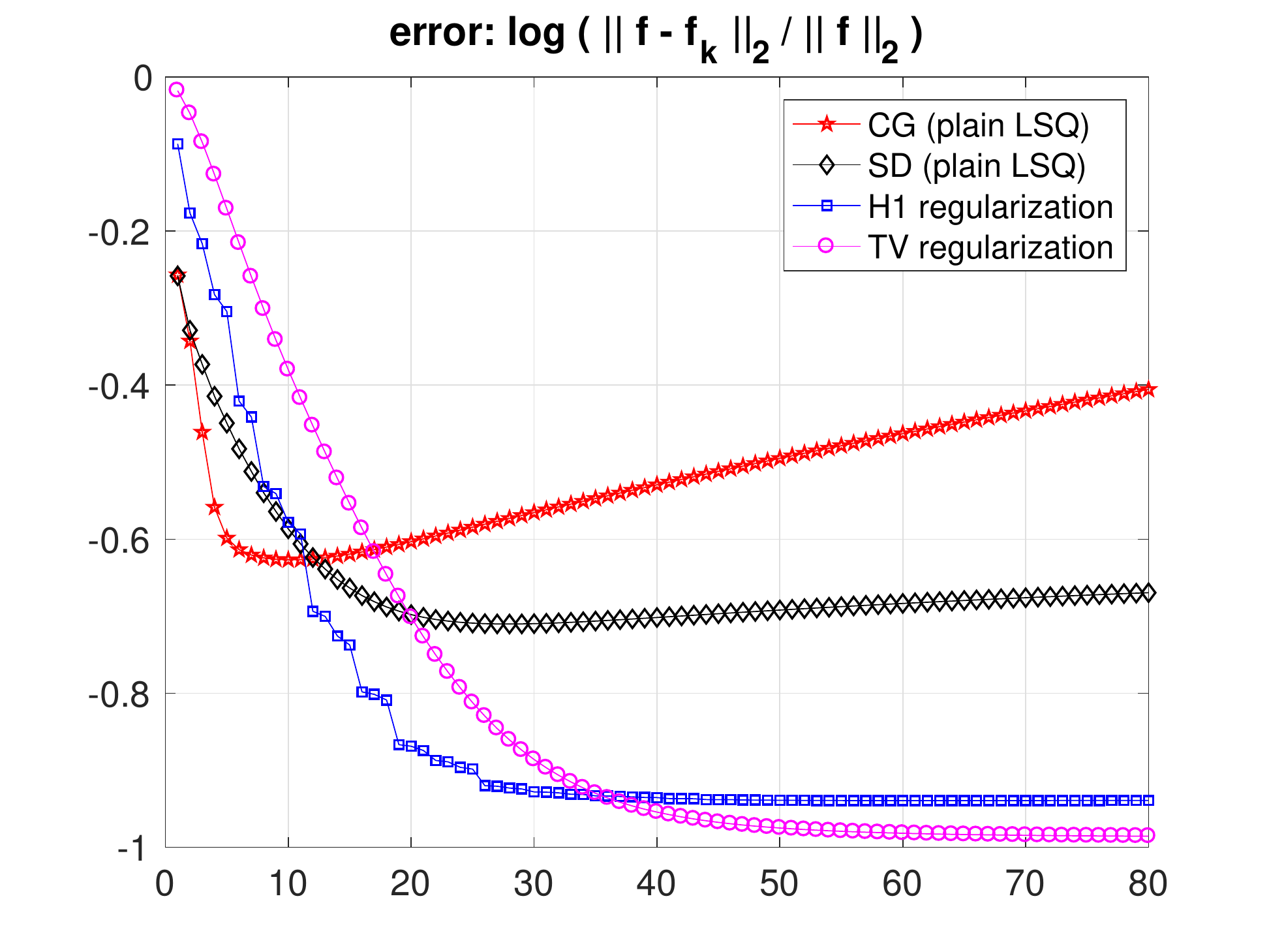} 
\includegraphics[width=0.48\textwidth]{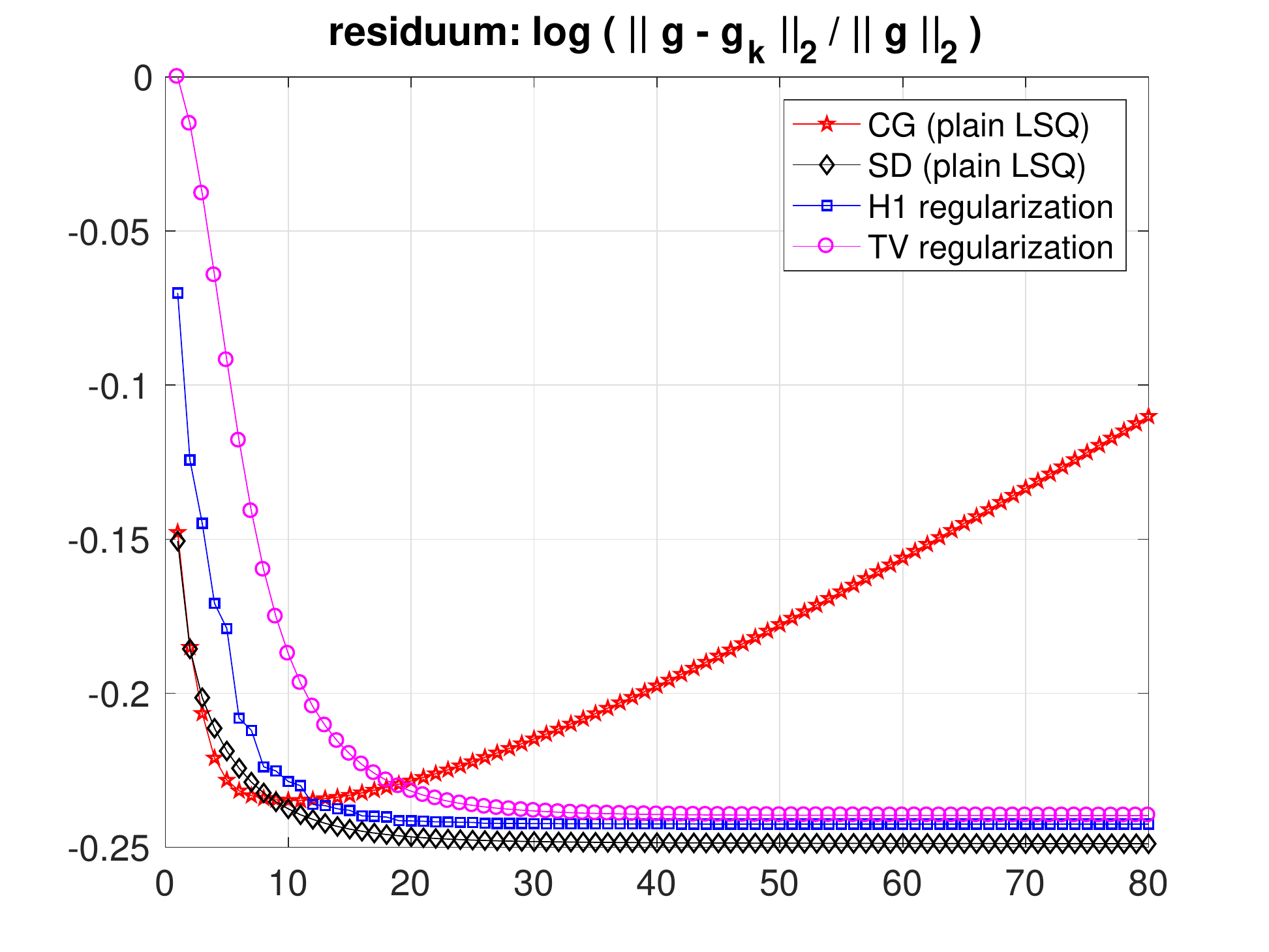}
\caption{Errors and residuals for the ill-posed partial data  case with noise. \label{fig:err2n}}
\end{figure}

\begin{figure}[tbh!]\centering
\includegraphics[width=\columnwidth]{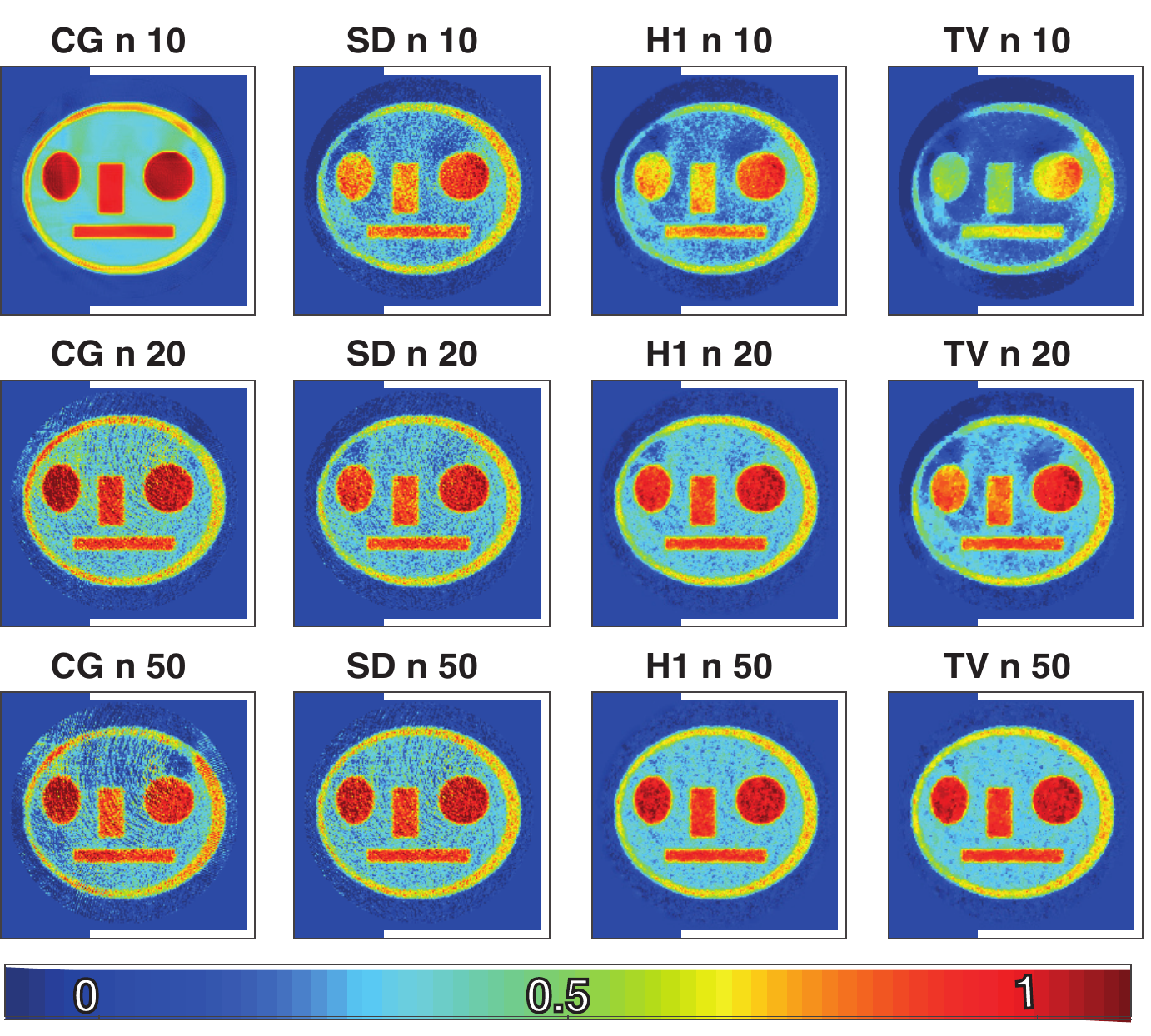}
\caption{Reconstructions after 10, 20 and 50 iterations for the ill-posed partial data  case with noise.\label{fig:rec2n}}
\end{figure}

\paragraph{Noisy data:}

The  methods from above are again applied, now to noisy data
with relative $\ell^2$-error about $59.7\%$.    The standard (unpenalized) iterative methods
 provide a regularization method when  combined with  early stopping.  In contrast, the  $H^1$- and TV-regularization methods converge to the minimizers of the corresponding Tikhonov functionals. Reconstruction results are shown in Figures~\ref{fig:err2n} and \ref{fig:rec2n}. In terms of reconstruction quality, TV-minimization is  the best method, followed by $H^1$-regularization.  The CG methods again behaves unstably and worse than the steepest descent method.      

The relative $\ell^2$-reconstruction  error  after  $50$ iterations  for the CG iteration, the steepest descent iteration, $H^1$-regularization and TV-regularization are  respectively $32 \%$,    $20.3\%$,    $11.5 \%$,  and $10.59 \%$. The corresponding residuals are  
$66.4 \%$,   $56.4 \%$,    $57.2 \%$,  and $57.6 \%$.

\section*{Acknowledgements}	

Linh Nguyen's research is partially supported by the NSF grants DMS 1212125 and DMS 1616904. Markus Haltmeier acknowledges support of the Austrian Science Fund (FWF), project P 30747-N32.

\appendix

\section{Appendix}
\subsection{Existence and uniqueness of adjoint equation} \label{A:well}
In this section, we prove the existence and uniqueness for the adjoint equation. Namely, consider the equation:
\begin{eqnarray} \label{E:Ap}
\left\{\begin{array}{ll} [c^{-2} \, \partial_{tt} + a \, \partial_t - \, \Delta] q =0,
& \mbox{for } (x,t) \in (\R^d \setminus \partial \Om) \times (0,T), \\[6 pt] q(0) =0, \quad q_t(0) =0, \\[6 pt]
\big[ q \big] =0, \Big[\frac{\partial q}{\partial \nu} \Big]= g. \end{array} \right.
\end{eqnarray}
\begin{definition} A function $q$ is a weak solution of (\ref{E:Ap}) if
\begin{itemize}
\item[i)]  $q \in L^2([0,T]; H^1(\R^d))$, $q' \in L^2([0,T];L^2(\R^d))$, $q'' \in L^2([0,T];H^{-1}(\R^d))$,
\item[ii)] $q(0) =0$ and $q_t(0) =0$, and
\item[iii)] for any function $\phi \in H_0^1(\R^d)$:
\begin{multline*} \int_{\R^d} c^{-2}(x) \, q_{tt}(x,t) \,\phi(x) \,dx  +   \int_{\R^d} a(x) \, q_{t}(x,t) \, \phi(x) \,dx  \\ + \int_{\R^d} \nabla q(x,t) \, \nabla \phi(x)\,dx = - \int_{\partial \Om}  \, g(y,t) \, \phi(y) \, dy,  \quad \mbox{a.e.~} t \in [0,T].
\end{multline*}
\end{itemize}
\end{definition} 
Let us note that from the above variational formulation, (\ref{E:Ap}) can be formally rewritten as the nonhomogeneous wave problem
\begin{eqnarray*}
\left\{\begin{array}{ll} [c^{-2} \, \partial_{tt} -a \, \partial_t - \, \Delta] q = -  \delta_{\pd \Om} \,  g , & \mbox{ on }  \R^d \times(0,T), \\[6 pt] q(0) =0, \quad q_t(0) =0, & \mbox{ on } \R^d. \end{array} \right.
\end{eqnarray*}
This formulation will be used for numerical simulation in Section~\ref{app:kspace}. Here are some results for equation (\ref{E:Ap}): 

\begin{theorem} \label{E:well-posed}
For any $$g \in L^2([0,T]; H^{1/2}(\partial \Om)) \cap H^1([0,T];H^{-1/2}(\partial \Om)),$$  equation (\ref{E:Ap}) has a unique weak solution. Moreover, 
\begin{itemize}
\item[i)] $q$ satisfies the finite speed of propagation property. Namely, let $c_+ := \max_{x \in \R^d} c(x)$, then $q(x,t) =0$ for any $(x,t) \in \Om^c \times [0,T]$ such that $dist(x,\partial \Om) \geq c_+ t$. 
\item[ii)] The following estimate holds \begin{equation} \label{E:est} \int_0^T \left[ \|q_t(t)\|^2 + \|q(t) \|^2_{H^1(\R^d)} \right] dt \ \leq C \|g'\|^2_{H^1([0,T];H^{-1/2}(\partial \Om))}. \end{equation}
Here, for simplicity, we use $\|\edot\|$ for the weighted $L^2$-norm with the weight $c^{-2}(x)$:
$$\|q_t(t)\|^2 = \int_{\R^d} c^{-2}(x) \, q^2_t(x,t) dx. $$  
\end{itemize}
\end{theorem}

\begin{proof}

Let $B_R$ denote the ball of radius $R$ centered at the origin and $R:=R_0 + c_+ T$, where $R_0$ satisfies $\Om \subset B_{R_0}$. Let $H_0^1(B_R)$ be the closure of $C_0^\infty(B_R)$ with respect to the norm 
$$\|f\|_{H_0^1(B_R)} = \left [ \int_{B_R} |\nabla f|^2 dx \right]^{1/2}.$$
Our proof is divided into two steps:

\emph{Step 1:} There exists a weak solution $q$ of (\ref{E:Ap}) on $B_R$. That is,  \begin{itemize}
\item[i')] $q \in L^2([0,T];H_0^1(B_R))$, $q' \in L^2([0,T];L^2(B_R))$, $q'' \in L^2([0,T];H^{-1}(B_R))$,
\item[ii')] $q(0)=0$ and $q'(0) =0$, and
\item[iii')] for any function $\phi \in H_0^1(B_R)$ 
\begin{multline*} \int_{B_R} c^{-2}(x) \, q_{tt}(x,t) \, \phi(x) \,dx  +   \int_{B_R} a(x) \, q_{t}(x,t) \, \phi(x) \,dx  \\ +  \int_{B_R} \nabla q(x,t) \, \nabla \phi(x) \,dx  = - \int_{\partial \Om}  \, g(y,t) \, \phi(y) \, dy \,\quad   \mbox{a.e } t \in [0,T]
\end{multline*}
\end{itemize}

\emph{Step 2:} The solution $q$ in Step~1 satisfies: $q(x,t) =0$ for all $(x,t) \in \Om^c \times [0,T]$ such that $dist(x,\partial \Om) \geq c_+ t$. 

Once both steps are proved, the solution $q$ of equation (\ref{E:Ap}) is just the trivial extension of $q$ into $[0,T] \times \R^d$. Let us now proceed to prove those steps. 

\emph{Proof of Step 1:} Let $\{\phi_{k}\}_k$ be an orthogonal basis of $H_0^1(B_R)$.\footnote{One such basis is the set of normalized eigenvectors of the Laplacian with the zero boundary condition.} For any integer $N$, we define $$q_N(x,t) = \sum_{i=1}^N d_i(t) \phi_i(x)$$ to be a solution of the system
\begin{multline}\label{E:key} \int_{B_R} c^{-2}(x) \, q_{N,tt}(x,t) \, \phi_i(x) \,dx  + \int_{B_R} a(x) \, q_{N,t}(x,t) \, \phi_i(x) \,dx   \\ +\int_{B_R} \nabla q_N(x,t) \, \nabla \phi_i(x) \,dx = -  \int_{\partial \Om} g(y,t) \, \phi_i(y) \, dy, \quad i =1,\dots,N.
\end{multline} together with the initial condition $q_N(x,0) = q_{N,t}(x,0) =0$. Since the above system is a standard linear ODE system for $(d_1,\dots, d_N)$, $q_N$ uniquely exists. Multiplying each equation by $d_i'(t)$ and summing them up, we obtain:
\begin{multline*} \int_{B_R} c^{-2}(x) q_{N,tt} (x,t) q_{N,t}(x,t) \,dx  +  \int_{B_R} a(x) \, [q_{N,t}(x,t)]^2 \,dx   \\ +\int_{B_R} \nabla q_N(x,t) \, \nabla q_{N,t} \,dx = -  \int_{\partial \Om} g(y,t) \, q_{N,t} (y,t) \, dy.
\end{multline*}
This implies
\begin{equation*} \frac{1}{2} \frac{d}{dt} \left[\int_{B_R} c^{-2}(x) |q_{N,t}(x,t)|^2  \,dx  +  \int_{B_R} |\nabla q_N(x,t)|^2 dx \right]   \leq  -  \int_{\partial \Om} g(y,t) \, q_{N,t}(y,t) \, dy.\end{equation*}
Taking the integration of both sides with respect to $t$ and using the initial conditions for $q_N$:
\begin{multline*} \frac{1}{2}  \left[\|q_{N,t}(\edot, t)\|^2 + \|q_N(\edot,t)\|^2_{H_0^1(B_R)} \right ]  \leq   \\ - \int_{\partial \Om} g(y,t) \, q_N(y,t) \, dy +  \int_0^t  \int_{\partial \Om} g_t(y,t) \, q_N(y,t) \, dy. \end{multline*}
Bounding the first term of the right hand side, we obtain
\begin{multline*}  \frac{1}{2}  \left[\|q_{N,t}(\edot,t)\|^2 + \|q_N(\edot,t)\|^2_{H_0^1(B_R)} \right ] \leq \|g(\edot,t)\|_{H^{-1/2}(\partial \Om)} \|q_N(\edot,t)\|^2_{H^{1/2}(\partial \Om)}  \\ + \int_0^t \|g_t(\edot,t)\|^2_{H^{-1/2}(\partial \Om)}  + \int_0^t \|q_N(\edot,t)\|_{H^{1/2}(\partial \Om)}.\end{multline*} 
Now, Young's inequality gives
\begin{multline*}  \frac{1}{2}  \left[\|q_{N,t}(\edot,t)\|^2 + \|q_N(\edot,t)\|^2_{H_0^1(B_R)} \right ]  \leq A \|g(\edot,t)\|^2_{H^{-1/2}(\partial \Om)} \\ + \frac{1}{2A}\|q_N(\edot,t)\|^2_{H^{1/2}(\partial \Om)}  + \int_0^t \|g_t(\edot,t)\|^2_{H^{-1/2}(\partial \Om)}  + \int_0^t \|q_N(\edot,t)\|_{H^{1/2}(\partial \Om)},\end{multline*} where $A>0$ can be any constant, whose value will be specified later. 
Noting that $\|q_N(\edot,t)\|_{H^{1/2}(\partial \Om)} \leq C \|q_N(\edot,t)\|_{H_0^{1}(B_R)}$ we obtain by choosing $A$ big enough 
\begin{multline*}  \frac{1}{2}  \left[\|q_{N,t}(\edot,t)\|^2 + \|q_N(\edot,t)\|^2_{H_0^1(B_R)} \right ]  \leq A \|g(\edot,t)\|^2_{H^{-1/2}(\partial \Om)} \\ + \frac{1}{4} \|q_N(\edot,t)\|^2_{H^{1}_0(B_R)}  + \int_0^t \|g_t(\edot,t)\|^2_{H^{-1/2}(\partial \Om)}   + C \int_0^t \|q_N(\edot,t)\|^2_{H^{1}_0(B_R)}. \end{multline*}
Here and in the sequel, $C$ is a generic constant whose value may vary from one place to another. Therefore, 
\begin{multline*} \|q_{N,t}(\edot,t)\|^2  +  \|q_N(\edot,t)\|^2_{H_0^1(B_R)}  \leq C \big( \|g(\edot,t)\|^2_{H^{-1/2}(\partial \Om)}  \\ +  \int_0^T \|g_t(\edot,t)\|^2_{H^{-1/2}(\partial \Om)} +  \int_0^t \|q_N(\edot,t)\|^2_{H_0^{1}(B_R)}\big), \quad t \in [0,T].\end{multline*}
Let $E_N(t) := \int_0^t \|q_{N,t}(\edot,t)\|^2 + \|q_N(\edot,t)\|^2_{H^{1}(B_R)}$. We arrive at
\begin{equation*}  E_N'(t) - C E_N(t)  \leq C \big( \|g(\edot,t)\|^2_{H^{-1/2}(\partial \Om)}   + \|g_t\|^2_{L^2([0,T],H^{-1/2}(\partial \Om))} \big),~ t \in [0,T].\end{equation*}
From the Grownwall's inequality, we obtain \begin{equation} \label{E:estimate} E_N(T) \leq C(  \|g\|^2_{L^2([0,T],H^{-1/2}(\partial \Om))}  +  \|g_t\|^2_{L^2([0,T],H^{-1/2}(\partial \Om))} ).\end{equation} Since $C$ is a constant independent of $N$, $\{q_{N}\}$ and $\{q_{N,t}\}$ are bounded sequences in $L^2([0,T],H^1_0(B_R))$ and $L^2([0,T];L^2(B_R))$, respectively. After possibly passing over to subsequences, we obtain
$q_N \rightharpoonup q$ in $L^2([0,T];H^1_0(B_R))$ and $q_{N,t} \rightharpoonup q_1$ in $L^2([0,T];L^2(B_R))$. It is easy to show that $q_1=q'$. Since $\{\phi_k\}$ is a basis of $H_0^1(B_R)$, from (\ref{E:key}), we obtain for any $v \in L^2([0,T];H_0^1(\Om))$:
\begin{multline*} \lim_{N \to \infty} \int_0^T \int_{\R^d} c^{-2}(x) \, q_{N,tt}(x,t) \, v(x,t) \,dx dt   + \int_0^T \int_{\R^d} a(x) \, q_{t}(x,t) \, v(x,t) \,dx  dt   \\ + \int_0^T \int_{\R^d} \nabla q(x,t) \, \nabla v(x,t) \,dx = -  \int_{\partial \Om} g(y,t) \, v(y,t) \, dy.
\end{multline*}
That is, $q_{N,tt}$ converges to an element in $L^2([0,T], H^{-1}(B_R))$. That is, $q_{tt} \in L^2([0,T], H^{-1}(B_R))$ and 
\begin{multline*} \int_0^T \int_{\R^d} c^{-2}(x) \, q_{tt}(x,t) \, v(x,t) \,dx \,dt + \int_0^T \int_{\R^d} a(x) \, q_{t}(x,t) \, v(x,t) \,dx  dt  \\ + \int_0^T \int_{\R^d} \nabla q(x,t) \, \nabla v(x,t) \,dx dt   = -  \int_0^T \int_{\partial \Om} g(y,t) \, v(y,t) \, dy \, dt.
\end{multline*}
Let $\phi \in H_0^1(B_R)$. For any $t_0 \in (0,T)$, choosing\footnote{For any set $U$, $\chi_U$ is the characteristic function of $U$.} $v(x,t) = \phi(x) \chi_{[t_0-\eps, t_0+ \eps]}(t)$, we obtain
\begin{multline*} \int_{t_0-\eps}^{t_0+\eps} \int_{\R^d} c^{-2}(x) \, q_{tt}(x,t) \, \phi(x) \,dx \,dt  + \int_{t_0-\eps}^{t_0+ \eps} \int_{\R^d} a(x) \, q_{t}(x,t) \, \phi(x) \,dx  dt  \\ + \int_{t_0-\eps}^{t_0+\eps} \int_{\R^d} \nabla q(x,t) \, \nabla \phi(x) \,dx dt = - \int_{t_0 - \eps}^{t_0+\eps}  \int_{\partial \Om} g(y,t) \, \phi(y) \, dy \, dt.
\end{multline*}
Dividing both sides by $2 \eps$ and send $\eps \to 0$, we obtain
\begin{multline*} \int_{B_R} c^{-2}(x) \, q_{tt}(x,t_0) \, \phi(x) \,dx   +   \int_{B_R} a(x) \, q_{t}(x,t_0) \, \phi(x) \,dx  \\ +  \int_{B_R} \nabla q(x,t_0) \, \nabla \phi(x) \,dx  = - \int_{\partial \Om}  \, g(y,t_0) \, \phi(y) \, dy \,\quad   \mbox{a.e } t_0 \in [0,T]
\end{multline*}
This finishes the proof of Step 1, since ii') easily follows from the fact that $q_N (\edot,0) =0$ and $q_{N,t} (\edot,0) =0$. 

\emph{Proof of step 2:} We first prove the result in the case $u' \in L^2([0,T],H^1(\Om))$ and $u'' \in L^2([0,T],L^2(\Om))$. Let $(x_0,t_0) \in (B_R \setminus \Om) \times [0,T]$ such that $dist(x_0,\partial \Om)> c_+ t_0$. There is $\eps_0>0$ such that for each $t \in [0,t_0]$, we have $B(x_0, (c_++ \eps_0) (t_0-t))  \cap \partial \Om =\emptyset$. We also denote $\mO_t = B(x_0, c (t_0-t)) \cap B_R$ and
$$E(t) = \frac{1}{2} \int_{\mO_t} c^{-2}(x) |q_t(x,t)|^2 + |\nabla q(x,t)|^2 dx, \quad 0 \leq t \leq t_0.$$
Then, 
\begin{multline*} \frac{d}{dt}E(t) = -\frac{c_+}{2} \int_{\partial \mO_t \setminus \partial B_R} c^{-2}(x) |q_t(x,t)|^2 + |\nabla q(x,t) |^2 d\sigma(x) \\+ \int_{\mO_t} c^{-2}(x) q_t(x,t) \, q_{tt}(x,t) + \nabla q(x,t) \nabla q_t(x,t) \, dx.\end{multline*}
Taking integration by parts for the second integral gives the following formula of $\frac{d}{dt}E(t)$:
\begin{multline*}  -\frac{c_+}{2} \int_{\partial \mO_t \setminus \partial \Om_R} \big[ c^{-2}(x) |q_t(x,t)|^2 + |\nabla q(x,t)|^2  - 2 \partial_\nu q(x,t)  \\ \times \frac{q_t(x,t)}{c_+} \big] d\sigma(x)  + \int_{\mO_t} \big[c^{-2}(x)  q_{tt}(x,t) - \Delta q(x,t) \big] q_t(x,t) \, dx. \end{multline*}
Noting that the integrand of the first term on the right hand side is nonnegative, we arrive to
\begin{equation*} \frac{d}{dt}E(t)  \leq  \int_{\mO_t} \big[c^{-2}(x)  q_{tt}(x,t) - \Delta q(x,t) \big] q_t(x,t) \, dx. \end{equation*}
Let us recall that for any function $\phi \in H_0^1(B_R)$ 
\begin{multline*} \int_{B_R} c^{-2}(x) \, q_{tt}(x,t) \, \phi(x) \,dx   +   \int_{B_R} a(x) \, q_{t}(x,t) \, \phi(x) \,dx   \\ +  \int_{B_R} \nabla q(x,t) \, \nabla \phi(x) \,dx  = -\int_{\partial \Om}  \, g(y,t) \, \phi(x) \, dy.
\end{multline*}
For $0<\eps<\eps_0$ we choose $\varphi_\eps \in C^\infty(\R^d)$ be a nonnegative function such that $\varphi \equiv 1$ on $B_{(x_0, c_+(t_0 -t))}$ and $\varphi \equiv 0$ outside of $B_{(x_0,(c_+ + \eps) (t_0-t))}$ and $\lim_{\eps \to 0} \varphi_\eps = \chi_{B_{x_0,c_+(t_0-t)}}$ on $L^2(\R^d)$. Choosing $\phi(x)= q_t(x,t) \varphi_\eps(x)$, we obtain
\begin{multline*} \int_{B_R} c^{-2}(x) \, q_{tt}(x,t) \, q_t(x,t) \varphi_\eps(x) \,dx    +   \int_{B_R} a(x) \, q_{t}(x,t) \, q_t(x,t) \varphi_\eps (x) \,dx   \\ +  \int_{B_R} \nabla q(x,t) \, \nabla [v_t(x,t) \varphi_\eps(x)] \,dx  =0.
\end{multline*}
Taking integration by parts for the last integral and combine it with the first integral, we obtain
\begin{equation*} \int_{B_R} \left[c^{-2}(x) \, q_{tt}(x,t) - \Delta q(x,t)\right] \, q_t(x,t) \varphi_\eps(x) \,dx    +   \int_{B_R} a(x) \, q^2_{t}(x,t)  \varphi_\eps (x) \,dx  =0.
\end{equation*}
Therefore, 
\begin{equation*} \int_{B_R} \left[c^{-2}(x) \, q_{tt}(x,t) - \Delta q(x,t)\right] \, q_t(x,t) \varphi_\eps(x) \,dx  \leq 0.
\end{equation*}
Taking the limit as $\eps \to 0$, we obtain 
\begin{equation*} \int_{\mO_t} \left[c^{-2}(x) \, q_{tt}(x,t) - \Delta q(x,t)\right] \, q_t(x,t)  \,dx  \leq 0.
\end{equation*}
 We obtain $\frac{E(t)}{dt} \leq 0.$ Noting that $E(0) =0$, we arrive at $E(t) =0$ for all $t \in [0,t_0]$. Therefore, $q(x,t) = 0$ on $\mO_t$ for all $t \in [0,t_0]$. Since this is correct for all $(x_0,t_0) \in \Om^c \times [0,T]$ such that $dist(x_0, \partial \Om) > c_+ t_0$,  It is now easy to see $q(x,t) = 0$ for all $(x,t) \in \Om^c$ such that $dist(x,\partial \Om) \geq c_+ t$. \\
In general, we do not have the required regularity for the above proof. However, consider $Q(x,t)  = \int_0^t q(x,\tau) d \tau$. Then, $Q$ satisfies the same equation (with a different jump function) and the required regularity. The above proof then shows that $Q(x,t) =0$ for all $(x,t) \in \Om^c \times [0,T]$ such that $dist(x,\partial \Om) \geq c_+ t$. It implies the same result  for $q(x,t)$.  This finishes proof of Step 2. \\
\emph{Finishing the proof:} Now extending $q$ into $\R^d \times [0,T]$ by zero on $(\R^d \setminus B_R) \times [0,T]$, we can easily prove that $q$ is a weak solution on $\R^d \times [0,T]$. Moreover, $q$ satisfies the finite speed of propagation (i). Finally, the estimate (\ref{E:est}) follows from (\ref{E:estimate}). The uniqueness of $q$ is simple (see, e.g., proof of Theorem A.2 in \cite{belhachmi2016direct}), we leave the details to the reader. 

\end{proof}

\subsection{A $k$-space method for the damped wave equation}
\label{app:kspace}

In this subsection, we briefly describe the $k$-space method as we use it to
numerically compute the solution of the wave equation,  which is required for evaluating the forward operator $\Wo$ and its adjoint $\Wo^*$.
For the case $a=0$, several methods for numerically solving the underlying acoustic wave  equation have been used in  PAT. This includes finite difference methods \cite{Burgholzer2007,nguyen2016dissipative,StefanovYang2},
finite element methods \cite{belhachmi2016direct} as well as  Fourier spectral and $k$-space methods \cite{cox2007k,huang2013full,treeby2010k}.
We now extend the $k$-space method  to the case $a \neq 0$ because this method does not suffer from numerical dispersion \cite{compani1986k}.

Consider the solution $p \colon \R^d  \times (0, T) \to \R$ of the
damped wave equation
 \begin{align} \label{eq:waveS}
	& [ c^{-2} \, \partial_{tt} + a  \, \partial_{t}  -   \Delta] p   = s
	&&  \mbox{ on } \R^d \times (0,T)  \,,
	\\ \label{eq:waveSi}
 	&
	p(0)  = f
	&&  \mbox{ on }  \R^d  \,,
	\\  \label{eq:waveSii}
 	&
	p_t(0)  = - c^2 \, a \, f
	&&  \mbox{ on }  \R^d  \,.
\end{align}
\mbox
Here, $s \colon \R^d \times   (0,T) \to \R$ is a given source term
and $f  \colon \R^d   \to \R$ the given initial pressure.
To derive the  $k$-space method one first rewrites \eqref{eq:waveS} in the form
 \begin{equation} \label{eq:waveS2}
  [\partial_{tt}    - c_0^{2}   \Delta ] p  =
	(1 - c_0^{2} / c^{2})   p_{tt}- c_0^{2} a \,  p_{t} +  c_0^{2} s
\end{equation}
where  $c_0>0$ is a suitable constant; we take $c_0 =c_+ := \max \set{ c(x) \colon x \in \R^2}$.

The $k$-space method is  derived from  \eqref{eq:waveS2} by
introducing the auxiliary functions  $v(x,t)$ and $r(x,t)$ such that
$v_{tt}(x,t) =  (1 - c_0^{2} / c^{2}(x)  )   p_{tt}(x,t) $ and  $r_{tt}(x,t) = c_0^{2} a(x)   p_{t}(x,t) $. Such an approach shows that  \eqref{eq:waveS2} is
equivalent to the following system of equations,
\begin{align}  \label{eq:waveK1}
        [\partial_{tt} -  c_0^{2} \Delta] w
	& = c_0^2 \, s + c_0^2
	\Delta v - c_0^2 \, \Delta r \,,
\\	 \label{eq:waveK2}
	v &=  \kl{ c^2 / c_0^2 - 1 }  \, (w - r)
\\  \label{eq:waveK3}
	p &=  v+ w - r
\\  \label{eq:waveK4}
	r(t) &= c_0^2 a   \int_{0}^t  p(s) \rmd s  \,.
\end{align}
Interpreting $c_0^2
	\Delta v(x,t) - c_0^2 \, \Delta r(x,t) $ as an additional source term, \eqref{eq:waveK1} is a standard wave equation
with  constant sound speed $c_0$.
This suggests the time stepping  formula
\begin{multline} \label{eq:timestep}
w(x,t + h_t)
=2  w(x,t) -   w(x,t - h_t)
- 4  \Fo_\xi^{-1} \Bigl[  \sin(c_0 \sabs{\xi} h_t/2)^2 \times
\\
\Fo_x [w(x,t) + v(x,t) - r(x,t) ]
  -
(c_0h_t/2)^2   \sinc(c_0 \sabs{\xi} h_t/2)^2   \Fo_x [s(x,t) ]  \Bigr]
 \,,
\end{multline}
where $\Fo_x$ and $\Fo_\xi^{-1}$ denote  the Fourier and inverse Fourier transforms in the spatial variable $x$ and the spatial frequency variable $\xi$, respectively, and  $h_t > 0$ is a time stepping size.

The resulting $k$-space method for  solving~\eqref{eq:waveS}
is summarized in Algorithm~\ref{alg:kpace}.

\begin{alg}[The $k$-space\label{alg:kpace} method]
For given initial pressure $f(x)$ and source  term $s(x,t)$
approximate  the solution  $p(x,t)$ of \eqref{eq:waveS} as follows:
\begin{enumerate}[leftmargin=3em,label= (\arabic*)]
\item\label{k1}
Set $t = 0$ and define initial conditions
\begin{itemize}
\item $r(x,0)  = 0$;
\item $v(x,0)  = (1-c_0^2 / c^2 (x))   f(x)$;
\item $w(x,0) =  c_0^2 / c^2 (x)      f(x)$;
\item $w(x,-h_t) =  (1 + h_t c_0^2   a(x) ) w(x,0) $.
\end{itemize}

\item\label{k2}
Compute $w(x,t + h_t)$ by evaluating \eqref{eq:timestep};

\item\label{k3} Make the updates
 \begin{itemize}
 \item $v(x,t + h_t) \coloneqq   \kl{c^2(x)/c_0^2- 1} \, ( w(x,t+h_t) - r(x,t) ) $;
 \item $p(x,t + h_t) \coloneqq   v(x,t + h_t) + w(x,t + h_t) - r(x,t) $;
 \item  $r(x,t + h_t)  \coloneqq   r(x,t) + c_0^2 a(x) p(x,t + h_t)  h_t $;
 \end{itemize}

 \item\label{k7} Set $t \gets t+h_t $ and go back to \ref{k3}.

\end{enumerate}
\end{alg}

 Algorithm~\ref{alg:kpace} can directly be used to evaluate the
forward operator $\Wo f$ by taking  $s(x,t) =0$ and  restricting
the solution to the measurement surface $S_R$, that is $\Wo f = p|_{S_R \times (0, T)}$.
Recall that the  adjoint operator is given by $\Wo^*g  = q_t(0)$,
where $q\colon \R^2 \times (0,T) \to \R$ satisfies the adjoint wave equation
 \begin{align} \label{eq:wave2ad}
	&  [c^{-2} \, \partial_{tt} -    \Delta] q  = -    \delta_{S_R} \, g
	&&  \mbox{ on } \R^2 \times (0,T) \\
 	&
	q_t(T) = q(T)  = 0 && \mbox{ on } \R^d.
\end{align}
By substituting $t \gets T- t$  and taking $s(x,t)  =    g(x,T-t) \,  \delta_{S}(x)$
as source term in \ref{eq:waveS}, Algorithm~\ref{alg:kpace}
can also be used to evaluate the $\Wo^*$. In the partial data case where
measurements are made on a subset $S \subsetneq S_R$ only, the
adjoint  can be implemented by taking the source
 $s(x,t)  =  \chi(x,t) \,  g(x,T-t)\, \delta_{S_R}(x)$ with an
appropriate window function $\chi(x,t)$.
In order to   use all available data, in  our implementations
we  take  the window function to be equal to one on the observation
part $S$ and zero outside.  This choice of the window function is known to create streak artifacts
into the picture \cite{frikel2015artifacts,sima,BFNguyen}. However, as we see in our simulations, the artifacts fade
away quickly after several iterations when the problem is well-posed.

\end{document}